\documentclass[a4paper,11pt]{amsart}
\usepackage{amssymb}
\usepackage{amscd}
\usepackage{amsmath,amsthm}	
\usepackage{enumerate}
\usepackage{booktabs,multirow}
\usepackage{tikz}
\usepackage{rotating}
\usepackage{ytableau}
\usepackage[colorlinks=true]{hyperref}
\usetikzlibrary{patterns}
\usetikzlibrary{decorations.pathreplacing}
\usetikzlibrary{calc,through}
\usetikzlibrary{backgrounds}

\allowdisplaybreaks[1]
\numberwithin{figure}{section}
\numberwithin{equation}{section}

\title{Jones--Wenzl projections of type $D$ and Dyck tilings}
\author{Keiichi Shigechi}
\email{k1.shigechi AT gmail.com}
\date{\today}

\newcommand\tikzpic[2]{
\raisebox{#1\totalheight}{
\begin{tikzpicture}
#2
\end{tikzpicture}
}}

\newtheorem{theorem}[figure]{Theorem}
\newtheorem{example}[figure]{Example}
\newtheorem{lemma}[figure]{Lemma}
\newtheorem{defn}[figure]{Definition}
\newtheorem{prop}[figure]{Proposition}

\newtheorem{remark}[figure]{Remark}
\begin{document}
	
\begin{abstract}
We study the relation between a coefficient of an element of the Jones--Wenzl 
projection in the Temperley--Lieb algebra of type $D$ and an enumeration of Dyck tilings.
The coefficient can be non-recursively expressed as an enumerative generating function of Dyck tilings
by considering the generalized Hermite histories, which we call bi-colored vertical Hermite histories, on the tilings. 
\end{abstract}

\maketitle

\section{Introduction}
The Jones--Wenzl projection \cite{Jon83,Jon89,Wen87} is a special element of the 
Temperley--Lieb algebra \cite{TL71}.
This projection appears in many branches of mathematics and mathematical physics:
it gives a way to calculate the Jones polynomial of a knot \cite{Jon87,Jon89}, 
graphical calculus of the canonical bases in tensor products of the quantum group 
$U_{q}(\mathfrak{sl}_2)$ \cite{FK97}, 
and Soergel calculus \cite{EliWil16}.
In this paper, we study the Jones--Wenzl projections of type $D$ in terms of 
combinatorial objects called Dyck tilings.
We prove that a coefficient in the Jones--Wenzl projection of type $D$ can be non-recursively
expressed as the generating function of Dyck tilings.

The Jones--Wenzl projections of type $A$ were introduced by V.~F.~R.~Jones 
in \cite{Jon83}, and the recursive formula was given by H.~Wenzl in \cite{Wen87}.
This recursive formula was further simplified by I.~B.~Frenkel and M.~G.~Khovanov 
in \cite{FK97}, and a graphical interpretation of the same recursive formula 
was given by S.~Morrison in \cite{Mor17}.
In \cite{Shi24}, the author gave a combinatorial interpretation of the  
Jones--Wenzl projections of types $A$ and $B$. 
More concretely, the coefficient of an element in the projection is given 
by the generating function of combinatorial objects called Dyck tilings.
In \cite{Bai24}, J.~Baine gives another interpretation of the coefficients of the projection 
in terms of Kazhdan--Lusztig polynomials, equivalently, the graded ranks of Soergel modules,
for a wide class of types which includes types $A$ and $B$, but not type $D$.

A Dyck tiling is a tiling of a region between two Dyck paths by use of Dyck tiles.
Dyck tilings first appeared in the study of the Kazhdan--Lusztig polynomials for the 
Grassmannian permutations in \cite{SZJ12}, and independently in the study of double-dimer 
models \cite{KW11}.
We have two types of Dyck tilings (type $I$ and type $II$ in \cite{SZJ12}), and 
we make use of only cover-inclusive Dyck tilings (corresponding to type $I$ in \cite{SZJ12}) 
in this paper.
In \cite{KMPW12}, basic tools such as Dyck tiling strips, Dyck tiling ribbons, and Hermite histories, 
were established to study Dyck tilings. 
Among these tools, we generalize and make use of Hermite histories, which we call vertical Hermite 
histories, to give a combinatorial interpretation of the coefficients of elements in the Jones--Wenzl 
projection.

We start with the recursive relation for the Jones--Wenzl projection 
obtained by P.~Sentinelli in \cite{Sen19}, which is 
the type $D$  analogue of the Wenzl's recursive formula. 
Then, we simplify the recursive relation further along the spirit of 
Frenkel--Khovanov \cite{FK97} and Morrison \cite{Mor17}.
By applying the method used by Morrison, we obtain two types of recursive relations in terms of 
the diagrams according to the parity. 
One of the main results of this paper is that these simplified recursive formulas can be interpreted 
in terms of Dyck tilings by imposing a simple condition on them, or by introducing a new 
combinatorial description of Hermite histories which we call bi-colored vertical Hermite histories.
The advantage of our method is that a coefficient in the Jones--Wenzl projections is given by
an enumerative generating function of Dyck tilings.

It is well-known that the elements in the Temperley--Lieb algebra of type $A$ correspond 
to $n+1$-strand Temperley--Lieb diagrams which are diagrams with $n+1$ marked points on the 
top and the bottom, with $n+1$ non-intersecting strands connecting these marked points.
Similarly, the elements in the Tempereley--Lieb algebra of type $D$ correspond to 
the $n+1$-strand Temperley--Lieb diagrams with dots.
The diagrams are classified into two classes according to the parity of the number of dots in the diagram. 
The coefficients of the elements in the projection behave differently according to 
the parity. Therefore, we have two types of recursive formulas for the coefficients depending 
on the parity. We discuss these two cases in what follows.

The Temperley--Lieb algebra $\mathrm{TL}_{n+1}^{D}$ of type $D$ has the set of generators 
$\{E_0,E_1,\ldots,E_{n}\}$ with some relations.
If the number of dots in a diagram is even, this means that the algebraic representation of the diagram 
does not contain the product $E_{0}E_{1}$ of the two generators.
This condition is equivalent to considering the case where the product satisfies $E_0E_1=0$. 
Then, this case can be essentially reduced to the the type $B$ case (see e.g. \cite{StrWoj24}). 
Proposition \ref{prop:Deven} also shows this fact since the form of the recursive relation for the projections 
of type $D$ (under the condition $E_0E_1=0$) is the same as types $A$ and $B$ which are studied 
in \cite{Shi24}.
From these facts, the coefficients are expressed in terms of the generating functions of Dyck tilings 
which satisfy a simple condition (Theorem \ref{thrm:Deven}).

If a diagram contains an odd number of dots, the algebraic representation of the diagram 
contains the product $E_{0}E_{1}\neq 0$.
The existence of the product $E_0E_1$ is the main difficulty to obtain a combinatorial 
interpretation.
Below, we explain two difficulties coming from the product.

Firstly, in \cite{Gre98}, R.~M.~Green introduced a graphical representation of the Temperley--Lieb 
algebra of type $D$. In this representation, we have to introduce an equivalence of the 
diagrams corresponding to $E_1E_0$ and $E_{0}E_{1}$ which naively give different 
diagrams although we have $E_1E_0=E_0E_1$.
Instead of the equivalence relations, we introduce a similar, but more straightforward 
graphical representation.
When a diagram contains the product $E_0E_1$, we give a unique dot in the diagram.
This is well-defined since if the diagram contains the product $E_0E_1$, then 
it does not contain any other $E_0$ and $E_{1}$ in its algebraic representation.
We regard the unique dot as a landmark which indicates that the diagram 
contains the product $E_{0}E_{1}$. 
A diagram satisfies simple conditions (P1) and (P2) as in Section \ref{sec:grep}.

Secondly, if we derive the recurrence relation for a diagram with a unique dot, 
then this recurrence relation contains minus signs (Proposition \ref{prop:Dodd}).
The minus signs prevent us to make a combinatorial interpretation since we want 
to connect a coefficient in the projections with a generating function which 
enumerates combinatorial objects without minus signs.    
To overcome this point, we study the relation between the elements of type $A$ and type $D$.
If we take an appropriate linear combination of elements of type $D$, then this linear combination 
satisfies the recursive relation for type $A$ (Section \ref{sec:relAD}).
This allows us to rewrite the recursive relation into another recursive relation 
which does not have minus signs.
This new recursive relation naturally gives a combinatorial description of 
an element in the projection in terms of Dyck tilings satisfying simple 
conditions (Q1) to (Q4) in Section \ref{sec:DTodd}.

By combining the above two mentioned prescriptions, 
the recurrence relation becomes minus sign free, however, it involves the 
generating function of not only type $D$, but also type $A$ (Proposition \ref{prop:DinDAD}).
We introduce new combinatorial objects on a Dyck tiling, which we call bi-colored vertical Hermite histories.
The bi-colored vertical Hermite histories fit well to this situation, and give 
a combinatorial interpretation of the coefficients in the projections (Theorem \ref{thrm:DoddZ}).
A coefficient is expressed as a generating function of Dyck tilings 
using the bi-colored vertical Hermite histories.  
Especially, this combinatorial description is enumerative and non-recursive.

The paper is organized as follows. In Section \ref{sec:TLD}, we introduce the Temperley--Lieb 
algebra of type $D$, the graphical representation, and Jones--Wenzl projections.
We simplify the recursive relation of the projections, and give diagrammatic 
recursive relations in Section \ref{sec:rr}.
In Section \ref{sec:DT}, after introducing Dyck tilings and vertical Hermite histories,
we give a combinatorial interpretation of Jones--Wenzl projections in terms of 
bi-colored vertical Hermite histories.
Some explicit expressions of the coefficients in the Jones--Wenzl projections are 
given in Section \ref{sec:ex}. 
In Appendix \ref{app:A}, we give the explicit expression for the Jones--Wenzl 
projection of type $D$ of rank four.

\subsection*{Notation} 
We will use the standard notation for $q$-integers: $[n]:=(q^{n}-q^{-n})/(q-q^{-1})$.

\subsection*{Acknowledgement}
The author gratefully acknowledges Byung-Hak Hwang for collaboration at the 
early stage of this work and valuable comments on the manuscript.

\section{Temperley--Lieb algebra of type \texorpdfstring{$D$}{D}}
\label{sec:TLD}
\subsection{Definition}
The Dynkin diagram of type $D$ is depicted as follows:
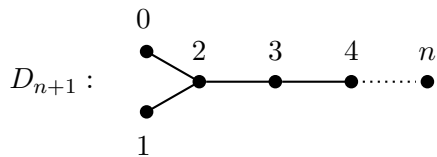
\begin{center}
  \begin{tikzpicture}[scale=.5	]
    \draw (-2.5,0) node[anchor=east]  {$D_{n+1}:$};
    \foreach \x in {0,...,3}
    \draw[xshift=\x cm,thick,fill] (\x cm,0) circle (.15cm);
    \foreach \x/\y in {0/2,1/3,2/4}
    \draw[xshift=\x cm,thick] (\x cm,0.3)node[anchor=south]{$\y$};
    \draw[xshift=3 cm,thick] (3 cm,0.3)node[anchor=south]{$n$};
    \draw[xshift=0 cm,thick,fill] (150: 16 mm) circle (.15cm);
    \draw[xshift=0 cm,thick,fill] (-150: 16 mm) circle (.15cm);
    \draw[dotted,thick] (4.3 cm,0) -- +(1.4 cm,0);
    \foreach \x in {0.09,1.09}
    \draw[xshift=\x cm,thick] (\x cm,0) -- +(1.7 cm,0);
    \draw[xshift=0 cm,thick] (150: 1.5mm) -- (150: 14.5 mm);
    \draw[xshift=0 cm,thick] (150:17mm)+(0,0.3)node[anchor=south]{$0$};
    \draw[xshift=0 cm,thick] (-150: 1.5mm) -- (-150: 14.5mm);
    \draw[xshift=0 cm,thick] (-150:17mm)+(0,-0.3)node[anchor=north]{$1$};
  \end{tikzpicture}
\end{center}
We take the convention that the Dynkin diagram of type $D_{n+1}$ is of 
rank $n+1$, and the labels of the nodes are $0,1,2\ldots,n$.
We have $D_1=A_1$, $D_2=A_1\times A_1$, and $D_3=A_3$.

The {\it Temperley--Lieb algebra} $\mathrm{TL}^{D}_{n+1}$ of type $D$ is 
a $\mathbb{C}(q)$-algebra with generators 
$\{E_0,E_1,\ldots,E_{n}\}$
which satisfy the following relations:
\begin{align*}
&E_{i}^{2}=-[2] E_{i}, \qquad\forall i\in\{0,1,\ldots,n\}\\
&E_{i}E_{j}E_{i}=E_{i},\qquad \text{if $i$ and $j$ are adjacent in the Dynkin diagram $D_{n+1}$}, \\
&E_{i}E_{j}=E_{j}E_{i},\qquad \text{if $i$ and $j$ are not adjacent in the Dynkin diagram $D_{n+1}$}.
\end{align*}
The Temperley--Lieb algebra $\mathrm{TL}^{A}_{n+1}$ of type $A$ is generated by the subset $\{E_1,E_2,\ldots, E_{n}\}$ or
$\{E_{0},E_2,\ldots, E_{n}\}$ of generators. 
It is well-known that the dimension of $\mathrm{TL}^{A}_{n}$ is given by the Catalan number 
$C_{n}:=\genfrac{}{}{}{}{1}{n+1}\genfrac{(}{)}{0pt}{}{2n}{n}=1,2,5,14,\ldots$.
Similarly, the dimension of $\mathrm{TL}^{D}_{n}$ is given by 
\begin{align}
\label{eq:TLdFCE}
\dim\mathrm{TL}^{D}_{n}=\genfrac{}{}{}{}{n+3}{2}C_{n}-1.
\end{align}
The dimension of $\mathrm{TL}^{D}_{n}$ is equal to the 
number of fully commutative elements of type $D$ studied in \cite{Fan96,Gre98,Ste96}.
In Section \ref{sec:DpTL}, we derive the formula (\ref{eq:TLdFCE}) by use of the 
correspondence between a Temperley--Lieb diagram and another diagram called 
a chord diagram.

Let $w:=(s_1,\ldots,s_{l})$ be a sequence of integers such that $s_{i}\in[0,n]$.
We call $E_{w}:=E_{s_1}\ldots E_{s_l}$ in $\mathrm{TL}^{D}_{n+1}$ a {\it word}.
A word $E_w$ is said to be a {\it reduced expression} if $E_{w}$ can not 
be written with less than $l$ generators.
The integer $l$ is the {\it length} of $E_{w}$.

\subsection{Graphical representation}
\label{sec:grep}
We first recall the graphical representation of $\mathrm{TL}^{A}_{n+1}$.
An $n+1$-strand Temperley--Lieb diagram is a diagram with $n+1$ stands such that 
it has $n+1$ marked points on the top and on the bottom, and $n+1$ arcs joining 
these marked points are non-intersecting. An arc joining two marked points 
on the top (resp. on the bottom) is called a {\it cup} (resp. a {\it cap}).
A cup $c$ is called inner-most cup if there is no smaller 
cup inside $c$.
A cup is called outer-most if there is no larger cup outside $c$ and no vertical
strands left to $c$.
An inner-most or outer-most cap is similarly defined.
We also say that a vertical strand which connects the right-most points 
on the top and bottom is an inner-most cap.
A vertical strand which is left-most in a diagram is said to be an outer-most cap.

An example of Temperley--Lieb diagram is given in Figure \ref{fig:TLdiagram}.
\begin{figure}[ht]
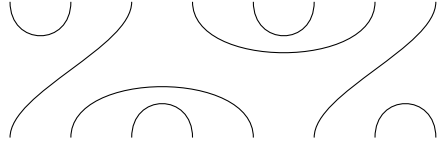

\tikzpic{-0.5}{[xscale=.8,yscale=.6]
\draw(0,3)..controls(0,2)and(1,2)..(1,3);
\draw(0,0)..controls(0,1)and(2,2)..(2,3);
\draw(1,0)..controls(1,1.5)and(4,1.5)..(4,0);
\draw(2,0)..controls(2,1)and(3,1)..(3,0);
\draw(3,3)..controls(3,1.5)and(6,1.5)..(6,3);
\draw(4,3)..controls(4,2)and(5,2)..(5,3);
\draw(5,0)..controls(5,1)and(7,2)..(7,3);
\draw(6,0)..controls(6,1)and(7,1)..(7,0);
}
\caption{An example of $8$-strand Temperley--Lieb diagram}
\label{fig:TLdiagram}
\end{figure}
This diagram has three caps and three cups. 
The cup connecting the first and second points is both 
inner-most and outer-most. The cup connecting the fourth and seventh 
points is neither inner-most nor outer-most. 
The cup connecting the fifth and sixth points is 
inner-most. We have two vertical strands which connect points on the 
bottom and on the top.
The left-most vertical strand is an outer-most cap.

Let $D_1$ and $D_2$ be $n+1$-strand Temperley--Lieb diagrams of type $A$.
Then, the product $D=D_1D_2$ of $D_1$ and $D_2$ is given by putting $D_1$ on 
the top of $D_2$.
If there is a loop in $D$, we give a factor $-[2]$ and delete the loop from the diagram.
For example, two diagrams $D_1$ and $D_2$ 
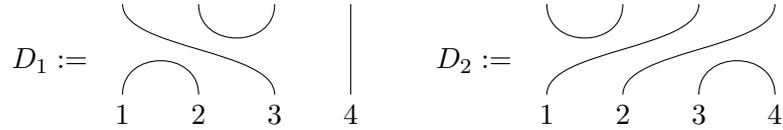
\begin{align*}
D_1:=\tikzpic{-0.5}{[yscale=0.6]
\draw(0,0)node[anchor=north]{$1$}..controls(0,1)and(1,1)..(1,0)node[anchor=north]{$2$}
(0,2)..controls(0,1)and(2,1)..(2,0)node[anchor=north]{$3$}
(1,2)..controls(1,1)and(2,1)..(2,2)(3,0)node[anchor=north]{$4$}--(3,2);
}\qquad
D_2:=\tikzpic{-0.5}{[yscale=0.6]
\draw(0,0)node[anchor=north]{$1$}..controls(0,1)and(2,1)..(2,2)
(1,0)node[anchor=north]{$2$}..controls(1,1)and(3,1)..(3,2)
(2,0)node[anchor=north]{$3$}..controls(2,1)and(3,1)..(3,0)node[anchor=north]{$4$}
(0,2)..controls(0,1)and(1,1)..(1,2);
}
\end{align*}
give the diagram $D=D_1D_2$:
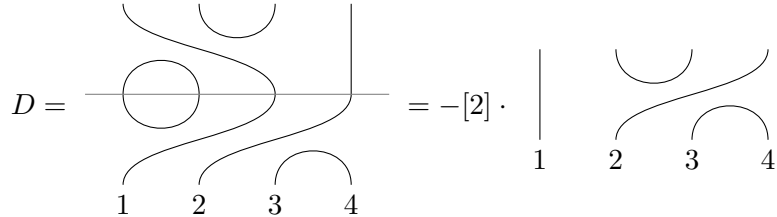
\begin{align*}
D=
\tikzpic{-0.5}{[yscale=0.6]
\draw(0,2)..controls(0,3)and(1,3)..(1,2)(0,4)..controls(0,3)and(2,3)..(2,2)
(1,4)..controls(1,3)and(2,3)..(2,4)(3,2)--(3,4);
\draw(0,0)node[anchor=north]{$1$}..controls(0,1)and(2,1)..(2,2)
(1,0)node[anchor=north]{$2$}..controls(1,1)and(3,1)..(3,2)
(2,0)node[anchor=north]{$3$}..controls(2,1)and(3,1)..(3,0)node[anchor=north]{$4$}
(0,2)..controls(0,1)and(1,1)..(1,2);
\draw[gray](-0.5,2)--(3.5,2);
}=
-[2]\cdot 
\tikzpic{-0.5}{[yscale=0.6]
\draw(0,0)node[anchor=north]{$1$}--(0,2)(1,0)node[anchor=north]{$2$}..controls(1,1)and(3,1)..(3,2)
(1,2)..controls(1,1)and(2,1)..(2,2)
(2,0)node[anchor=north]{$3$}..controls(2,1)and(3,1)..(3,0)node[anchor=north]{$4$};
}
\end{align*}

To consider $\mathrm{TL}^{D}_{n+1}$, we generalize the notion of the $n+1$-strand
Temperley--Lieb diagrams.
For type $D$, we allow to put a ``dot" on a cup, a cap and a vertical strand. 
A diagram with dots satisfies the following properties:
\begin{enumerate}[(P1)]
\item The number of dots in a diagram is one, or an even non-negative integer.
A cap, a cup, or a vertical strand contains at most one dot.
\item We consider the following three cases:
\begin{enumerate}
\item The number of dots is even. A dot is on an outer-most cap, outer-most cup or 
a left-most vertical strand.
\item The number of dots is one. A dot is on a cap or on a vertical strand, which 
connects the left-most marked point on the bottom and another marked point.
\item The diagram consisting of $n+1$ vertical strands (without caps and cups) does not
have a dot.
\end{enumerate}
\end{enumerate}
We do not consider a diagram $D$ such that the number of dots in $D$ is an odd number greater than one.

Following \cite{Gre98}, we introduce a graphical representation for $\mathrm{TL}^{D}_{n+1}$.
The generator $E_{i}$, $1\le i\le n$ is depicted as 
\begin{align*}
E_{i}=
\tikzpic{-0.5}{[xscale=.8]
\draw(-1,0)node[anchor=north]{$1$}--(-1,1)(1,0)node[anchor=north]{$i-1$}--(1,1);
\draw(0,0.5)node{$\cdots$};
\draw(2,0)node[anchor=north]{$i$}..controls(2,.5)and(3,.5)..(3,0)node[anchor=north]{$i+1$};
\draw(2,1)..controls(2,.5)and(3,.5)..(3,1);
\draw(4.2,0)node[anchor=north]{$i+2$}--(4.2,1);
\draw(5.2,0.5)node{$\cdots$};
\draw(6.2,0)node[anchor=north]{$n+1$}--(6.2,1);
}
\end{align*}
and the generator $E_{0}$ is depicted as 
\begin{align*}
E_{0}=
\tikzpic{-0.5}{[xscale=.8]
\draw(0,0)node[anchor=north]{$1$}..controls(0,0.5)and(1,0.5)..(1,0)node[anchor=north]{$2$};
\draw(0.5,0.63)node{$\bullet$};
\draw(0,1)..controls(0,0.5)and(1,0.5)..(1,1);
\draw(0.5,0.37)node{$\bullet$};
\draw(2,0)node[anchor=north]{$3$}--(2,1);
\draw(3,0.5)node{$\cdots$};
\draw(4,0)node[anchor=north]{$n+1$}--(4,1);
}
\end{align*}
The identity $\mathbf{1}$ corresponds to the diagram with $n+1$ vertical strands without dots.
The property (P2c) implies that there is no diagram which has the same connectivity as $\mathbf{1}$
and has a dot.
Note that only the generator $E_0$ has two dots, and other generators do not have dots in the graphical
representation.
The generators $E_1$ and $E_0$ have the same connectivity of marked points, and we distinguish 
between $E_0$ and $E_1$ by the existence of dots.

In the case of type $A$, we have a graphical calculation of an element $E_{w}$ which may not
be a reduced expression. If $w$ is not a reduced expression, we use the algebraic relations $E_{i}E_{i\pm1}E_{i}=E_i$
and $E_i^2=-[2]E_{i}$ to obtain a reduced expression. The former relation does not change the connectivity of 
a diagram, but the latter relation produces a loop. The weight $-[2]$ of a loop corresponds to the factor 
in the relation.
In the case of type $D$, we impose the properties (P1) and (P2) on a diagram.
Due to these properties, we can not have a graphical calculation of $E_w$ which is not 
a reduced expression.
In what follows, we consider the graphical representation of a reduced expression 
in $\mathrm{TL}^{D}_{n+1}$.

An element of $\mathrm{TL}^{D}_{n+1}$ is given by a product of generators.
Suppose that a reduced expression $E_{w}:=E_{s_{1}}\ldots E_{s_{m}}$, $s_{i}\in[0,n]$ is written as a 
product $E_{w}=E_{w_1}E_{w_2}$ of two reduced expressions $E_{w_1}$ and $E_{w_2}$.
In other words, we have $l(w)=l(w_1)+l(w_2)$.
Let $D$, $D_1$ and $D_2$ be diagrams corresponding to $w$, $w_1$ and $w_2$.
Then, the diagram $D$ is obtained from $D_1$ and $D_2$ by putting $D_1$ on the top of $D_{2}$, 
and by applying the rule (P2b) to $D$ if $D$ has an odd number of dots.

The graphical representation of generators implies that we have a loop without dots 
if we have $E_{i}^2=-[2]E_i$, $i\ge1$, in its algebraic representation. As far as we consider 
a reduced expression, we do not have such a loop.
Similarly, we have no loops with two dots since such a loop appears 
if and only if the expression contains $E_{0}^2=-[2]E_0$. 
A vertical strand, a cup, or a cap contains two dots if $E_{w}$ contains 
the product $E_0E_2E_0$ in its algebraic representation. Since we consider  
only reduced expressions, we do not have such a vertical strand, a cup or a cap.
From these, it is clear that a vertical strand, a cup or a cap contains at most 
one dot.

We may have a loop with a single dot in the diagram $D$. 
Such a loop comes from the product of two generators $E_{0}$ and $E_{1}$ in the algebraic representation.
In this case, we simply delete the loop with a dot from the diagram.
Then, since a deletion of a loop with a dot implies that we have an odd number of dots in the diagram,
we change the position of the dots such that it satisfies the condition (P2b).
In fact, if a reduced expression $E_{w}$ contains the product $E_{0}E_{1}$ (or equivalently $E_{1}E_{0}$),
we have no other $E_0$ and $E_{1}$ in its algebraic representation.
This means that if the number of dots in a diagram is odd, then we have a unique dot which comes 
from the product $E_0E_1$. Therefore, the property (P2b) is well-defined.

In this way, we obtain a graphical representation $D$ of a reduced word $E_{w}$
in $\mathrm{TL}^{D}_{n+1}$.

\begin{remark}
Three remarks are in order:
\begin{enumerate}
\item
Consider the product $E_0E_1$ of two generators. By the defining relations, we have 
$E_{0}E_{1}=E_{1}E_{0}$. If we apply  a naive representation by diagrams, $E_0E_1$ 
gives a diagram with a dotted cup and a loop with a single dot.
On the other hand, $E_{1}E_{0}$ gives a diagram with a dotted cap and a loop 
with a single dot.
However, the two products $E_{1}E_{0}$ and $E_{0}E_1$ define the same element 
in $\mathrm{TL}^{D}_{n+1}$, we need to introduce an equivalence relation on 
diagrams for $E_1E_0$ and $E_{0}E_{1}$. This is achieved in \cite{Gre98}.
\item
In our convention, to distinguish a diagram with a single dot from diagrams 
with an even number of dots, we introduce the rule (P2b).
If an element $E'$ in $\mathrm{TL}^{D}_{n+1}$ contains the product $E_{0}E_{1}$, then
it is easy to see that $E'$ contains no other $E_{0}$ or $E_1$ if $E'$ is a reduced 
expression. This means that if the number of dots is odd in a diagram $D$, then 
$D$ contains a unique dot since we have the generator $E_0$ and we delete one of the 
two dots by deleting a loop with a single dot.
We regard a unique dot as a sign of the parity of the number 
of dots.
\item By the property (P2b), we depict
\begin{align*}
E_{0}E_{1}=
\tikzpic{-0.5}{[xscale=.8]
\draw(0,0)node[anchor=north]{$1$}..controls(0,0.5)and(1,0.5)..(1,0)node[anchor=north]{$2$};
\draw(0,1)..controls(0,0.5)and(1,0.5)..(1,1);
\draw(0.5,0.37)node{$\bullet$};
\draw(2,0)node[anchor=north]{$3$}--(2,1);
\draw(3,0.5)node{$\cdots$};
\draw(4,0)node[anchor=north]{$n+1$}--(4,1);
}.
\end{align*}
The product $E_0E_1$ has the same connectivity of the strands as $E_0$ and $E_1$, and 
has a unique dot on the unique cap.
\end{enumerate}
\end{remark}

\begin{example}
\label{ex:AB}	
The graphical representations of $A=E_{2}E_{0}E_1$ and 
$B=E_{0}E_2E_{3}E_1E_{2}E_{0}$ in $\mathrm{TL}^{D}_{4}$ are given 
by 
\begin{align*}
A=\tikzpic{-0.5}{[xscale=.8,yscale=0.6]
\draw(0,0)node[anchor=north]{$1$}..controls(0,.8)and(1,.8)..(1,0)node[anchor=north]{$2$};
\draw(0,2)..controls(0,1)and(2,1)..(2,0)node[anchor=north]{$3$};
\draw(1,2)..controls(1,1.2)and(2,1.2)..(2,2);
\draw(3,0)node[anchor=north]{$4$}--(3,2);
\draw(0.5,0.58)node{$\bullet$};
},
\qquad
B=\tikzpic{-0.5}{[xscale=.8,yscale=0.6]
\draw(0,0)node[anchor=north]{$1$}..controls(0,0.8)and(1,0.8)..(1,0)node[anchor=north]{$2$};
\draw(2,0)node[anchor=north]{$3$}..controls(2,0.8)and(3,0.8)..(3,0)node[anchor=north]{$4$};
\draw(0,2)..controls(0,1.2)and(1,1.2)..(1,2);
\draw(2,2)..controls(2,1.2)and(3,1.2)..(3,2);
\draw(0.5,0.58)node{$\bullet$}(2.5,0.58)node{$\bullet$};
\draw(0.5,1.42)node{$\bullet$}(2.5,1.42)node{$\bullet$};
}.
\end{align*}
\end{example}

\subsection{Jones--Wenzl projections}
We summarize the properties which characterize the Jones--Wenzl projection of types $A$ and $D$.
The Jones--Wenzl projection $P_{n}$, $0\le n$, is a unique element in $\mathrm{TL}^{A}_{n+1}$ such that 
\begin{enumerate}
\item $P_{n}\neq0$;
\item $P_{n}P_i=P_iP_{n}=P_n$ for $0\le i\le n$;
\item $E_{i}P_{n}=P_{n}E_i=0$ for $1\le i\le n$.
\end{enumerate}
These three conditions uniquely fix the element $P_{n}$ for $0\le n$.

Similarly, the Jones--Wenzl projection $Q_{n}$, $0\le n$, of type $D$ is a unique element 
in $\mathrm{TL}^{D}_{n+1}$ such that 
\begin{enumerate}
\item $Q_{n}\neq0$;
\item $Q_{n}Q_{i}=Q_{i}Q_{n}=Q_{n}$ for $0\le i\le n$;
\item $E_{i}Q_{n}=Q_{n}E_i=0$ for $0\le i\le n$.
\end{enumerate}
The second conditions for $P_n$ and $Q_{n}$ insure that the coefficient of the identity $\mathbf{1}$ 
is one.

\begin{example}
We give first few explicit expressions of $P_n$ and $Q_n$:
\begin{align*}
&P_{0}=1, \qquad P_{1}=1+\genfrac{}{}{}{}{1}{[2]}E_{1}, \\
&P_{2}=1+\genfrac{}{}{}{}{[2]}{[3]}(E_1+E_2)+\genfrac{}{}{}{}{1}{[3]}(E_1E_2+E_2E_1),
\end{align*}
and 
\begin{align*}
&Q_{0}=1+\genfrac{}{}{}{}{1}{[2]}E_{0}, \qquad
Q_{1}=1+\genfrac{}{}{}{}{1}{[2]}(E_0+E_1)+\genfrac{}{}{}{}{1}{[2]^2}E_{0}E_1, \\
&Q_{2}=1+\genfrac{}{}{}{}{[3]}{[4]}(E_0+E_1)+\genfrac{}{}{}{}{[2]^2}{[4]}E_2 
+ \genfrac{}{}{}{}{[2]}{[4]}(E_{0}E_2+E_2E_0+E_1E_2+E_2E_1) \\
&\qquad+\genfrac{}{}{}{}{[2]^3}{[4][3]}E_{0}E_{1}
+\genfrac{}{}{}{}{1}{[4]}(E_{0}E_2E_1+E_1E_2E_0) \\
&\qquad+\genfrac{}{}{}{}{[2]^2}{[4][3]}(E_0E_1E_2+E_2E_0E_1)
+\genfrac{}{}{}{}{[2]}{[4][3]}E_{2}E_{0}E_{1}E_{2}.
\end{align*}
Since $D_{3}=A_3$, the projection $P_{3}$ is expressed in terms 
of $Q_{2}$ by replacing $E_{0}$ with $E_{3}$. 

The explicit expression of $Q_3$ is given in Appendix \ref{app:A}.
\end{example}

\section{Recurrence relations for Jones--Wenzl projections}
\label{sec:rr}
\subsection{Type \texorpdfstring{$A$}{A}}
In \cite{Wen87}, Wenzl gave a recurrence relation for $P_{n}$:
\begin{align*}
P_{n}=P_{n-1}+\genfrac{}{}{}{}{[n]}{[n+1]}P_{n-1}E_{n}P_{n-1},
\end{align*}
with the initial condition $P_{0}=1$.

In \cite{FK97,Mor17}, the above recurrence relation is further 
simplified.
For this, we define the elements in $\mathrm{TL}^{A}_{n+1}$ 
by $g^{A}_{n,i}=E_{n}E_{n-1}\ldots E_{i}$ for $1\le i\le n$, and $g^{A}_{n,n+1}=1$.
Note that the diagram corresponds to $g^{A}_{n,n+1}$ is the diagram consisting of $n+1$ vertical
strands.
Then, we have the following proposition:
\begin{prop}[{\cite[Proposition 3.3]{Mor17}}]
The coefficient of $g^{A}_{n,i}$, $1\le i\le n+1$, in $P_{n}$ is $\genfrac{}{}{}{}{[i]}{[n+1]}$, and 
$P_n$ satisfies the recurrence relation:
\begin{align*}
P_{n}=P_{n-1}\left(\sum_{i=1}^{n+1}\genfrac{}{}{}{}{[i]}{[n+1]}g^{A}_{n,i}\right).
\end{align*}
\end{prop}

\subsection{Type \texorpdfstring{$D$}{D}}
To state the recurrence relation for $Q_{n}$, we introduce some notations.
We define the element $E_{w_n}$ in $\mathrm{TL}^{D}_{n+1}$ by 
\begin{align*}
E_{w_n}:=E_{n}E_{n-1}\cdots E_{2}E_{0}E_{1}E_{2}\cdots E_{n},
\end{align*} 
and define two coefficients
\begin{align*}
A_{n}:=
\begin{cases}
\displaystyle\genfrac{}{}{}{}{1}{[2]}, & \text{ if } n=1, \\[10pt]
\displaystyle\genfrac{}{}{}{}{[n][2n-2]}{[2n][n-1]}, & \text{ otherwise },
\end{cases}
\qquad
B_{n}:=\genfrac{}{}{}{}{[n]}{[2n][n+1]}.
\end{align*}

\begin{prop}[{\cite[Proposition 4.8]{Sen19}}]
\label{prop:recQ1}
The generalized Jones--Wenzl projection $Q_{n}$ satisfies the following recurrence relation:
\begin{align}
\label{eq:recQ1}
Q_{n}=Q_{n-1}+A_{n}Q_{n-1}E_{n}Q_{n-1}+B_{n}Q_{n-1}E_{w_n}Q_{n-1}.
\end{align}
\end{prop}

To simplify Eq. (\ref{eq:recQ1}) further, we first define 
some elements in $\mathrm{TL}^{D}_{n+1}$.
For $i=1,2,\ldots,n+1$ and $j=0,1,\ldots,n$,
we define 
\begin{align*}
&g_{n,i}:=\begin{cases}
1, & \text{ if } i=n+1, \\
E_{n}E_{n-1}\cdots E_{i}, & \text{ if } i\le n,
\end{cases}\\
&h_{n,j}:=\begin{cases}
E_{n}E_{n-1}\cdots E_{2}E_{0}, & \text{ if } j=0, \\
E_{n}E_{n-1}\cdots E_{2}E_{0}E_{1}E_2\cdots E_{j}, & \text{ if } j\ge1.
\end{cases}
\end{align*}
Note that we have $h_{n,n}=E_{w_n}$.

\begin{prop}
\label{prop:Q1}
The projections $Q_{n}$ satisfy the following recurrence relation:
\begin{align}
\label{eq:Q2}
Q_{n+1}=Q_{n}\left(\sum_{i=1}^{n+2}\mathrm{coef}(g_{n+1,i})g_{n+1,i}
+\sum_{j=0}^{n+1}\mathrm{coef}(h_{n+1},j)h_{n+1,j}\right),
\end{align}
where the two coefficients are given by 
\begin{align}
\label{eq:coefgh}
\mathrm{coef}(g_{n,i}):=
\begin{cases}
\displaystyle\genfrac{}{}{}{}{[n]}{[2n]}, & \text{ if } i=1, \\[10pt]
\displaystyle\genfrac{}{}{}{}{[n][2i-2]}{[2n][i-1]}, & \text{ otherwise},
\end{cases}
\qquad
\mathrm{coef}(h_{n,j}):=\genfrac{}{}{}{}{[n][n+1-j]}{[2n][n+1]},
\end{align}
for $1\le i\le n+1$ and $0\le j\le n$.
\end{prop}

Before proceeding with the proof of Proposition \ref{prop:Q1}, we write 
down the actions of $E_{n+1}$ and $E_{w_{n+1}}$ on $g_{n,i}$ and $h_{n,j}$.
\begin{lemma}
\label{lemma:Eongh}
We have 
\begin{align*}
&E_{n+1}g_{n,i}=g_{n+1,i}, \qquad \text{for } 1\le i\le n+1, \\
&E_{n+1}h_{n,j}=h_{n+1,j}, \qquad \text{for } 0\le j\le n, \\
&E_{w_{n+1}}g_{n,i}=h_{n+1,i}, \qquad \text{for }1\le i\le n+1, \\
&E_{w_{n+1}}h_{n,0}=h_{n+1,1}, \\
&E_{w_{n+1}}h_{n,j}=-[2]h_{n+1,j}, \qquad \text{for } 1\le j\le n.
\end{align*}
\end{lemma}
\begin{proof}
We calculate them by use of the definitions of $g_{n,i}$ and $h_{n,j}$, and the defining 
relations of $\mathrm{TL}^{D}_{n+1}$.
\end{proof}

\begin{proof}[Proof of Proposition \ref{prop:Q1}]
We prove the statement by induction on $n$ following the method developed in \cite{Mor17}. 
For $n=1$, one can show Eq. (\ref{eq:Q2}) by a simple calculation. 
We assume that Eq. (\ref{eq:Q2}) holds up to $n-1$.

Let $\mathcal{K}_{n}$ be the linear span of the elements $g_{n,i}$ and $h_{n,j}$,
and $\mathcal{K}^{\perp}_{n}$ the linear span of the other elements.
Then, an element $D\in\mathcal{K}^{\perp}_{n}$ can be written as 
$D=E_{j}D'$ with some $j\in\{0,1,2,\ldots,n-1\}$ and $D'$.
This implies that $Q_{n}E_{n+1}\mathcal{K}^{\perp}_{n}=0$ since 
we have $E_{n+1}E_{j}=E_{j}E_{n+1}$ and $Q_{n}E_{j}=0$ by definition of $Q_{n}$.
The projection $Q_{n}$ is decomposed as $Q_{n}=Q_{n}^{\mathcal{K}}+Q_{n}^{\mathcal{K}^{\perp}}$
with $Q_{n}^{\mathcal{K}}\in\mathcal{K}_{n}$ and $Q_{n}^{\mathcal{K}^{\perp}}\in\mathcal{K}^{\perp}_{n}$.
We have 
\begin{align*}
&Q_{n}E_{n+1}Q_{n}=Q_{n}E_{n+1}(Q_{n}^{\mathcal{K}}+Q_{n}^{\mathcal{K}^{\perp}})=Q_{n}E_{n+1}Q_{n}^{\mathcal{K}}, \\
&Q_{n}E_{w_{n+1}}Q_{n}=Q_{n}E_{w_{n+1}}Q_{n}^{\mathcal{K}}.
\end{align*}
By combining these together, we have
\begin{align}
\label{eq:Q2rec}
\begin{aligned}
Q_{n+1}&=Q_{n}+A_{n+1}Q_{n}E_{n+1}Q_{n}+B_{n+1}Q_{n}E_{w_{n+1}}Q_{n}, \\
&=Q_{n}+A_{n+1}Q_{n}E_{n+1}Q_{n}^{\mathcal{K}}+B_{n+1}Q_{n}E_{w_{n+1}}Q_{n}^{\mathcal{K}}.
\end{aligned}
\end{align}
Since $Q_{n}^{\mathcal{K}}\in\mathcal{K}_{n}$, $Q_{n}$ has an expression of the following form:
\begin{align}
\label{eq:Qansatz}
Q_{n}^{\mathcal{K}}=\sum_{i}\mathrm{coef}'(g_{n,i})g_{n,i}+\sum_{j}\mathrm{coef}'(h_{n,j})h_{n,j}.
\end{align}
with some coefficients $\mathrm{coef}'(g_{n,i})$ and $\mathrm{coef}'(h_{n,j})$.
From Eqs. (\ref{eq:Q2rec}) and (\ref{eq:Qansatz}), Lemma \ref{lemma:Eongh} and the induction hypothesis,
we obtain recursive formulas for $\mathrm{coef}'(g_{n,i})$, $\mathrm{coef}'(h_{n,j})$, $\mathrm{coef}(g_{n,i})$
and $\mathrm{coef}(h_{n,j})$.
These recursive formulas can be solved as 
\begin{align*}
\mathrm{coef}(g_{n,i})=\mathrm{coef}'(g_{n,i})=\genfrac{}{}{}{}{[n][2i-2]}{[2n][i-1]},\qquad
\mathrm{coef}(h_{n,j})=\mathrm{coef}'(h_{n,j})=\genfrac{}{}{}{}{[n][n+1-j]}{[2n][n+1]},
\end{align*}
which implies Eq. (\ref{eq:coefgh}).
By substituting these into Eq. (\ref{eq:Q2rec}) and rearranging the terms, we obtain
Eq. (\ref{eq:Q2}), which completes the proof.
\end{proof}

\subsection{Recurrence relation via diagrams}
\label{sec:rrd}
To derive the recurrence relation via diagrams, we first consider the 
graphical representation of the elements $g_{n,i}$ and $h_{n,j}$
with $1\le i\le n+1$ and $0\le j\le n$.
The element $g_{n,i}$, $1\le i\le n$ has a cup connecting the right-most point 
with the point next to it on the top, a cap connecting the $i$-th point 
with the $i+1$-th point on the bottom, and $n-1$ vertical strands.
The element $g_{n,n+1}$ consists of $n+1$ vertical strands.
Similarly, the element $h_{n,j}$, $1\le j\le n$, has a cup connecting
the right-most point with the point next to it on the top, a cap connecting 
the $j$-th point with the $j+1$-th point on the bottom, and $n-1$ vertical
strands. Further, a cap or a strand which contains the left-most point in the bottom
has a dot. 
The element $h_{n,0}$ has the same connectivity as $g_{n,1}$, and 
the unique cap and the left-most vertical strands have a dot. 
These elements $g_{n,i}$ and $h_{n,j}$ satisfy the properties (P1) and (P2).
In Figure \ref{fig:gh}, we depict the elements $g_{3,i}$ and $h_{3,j}$.
\begin{figure}[ht]
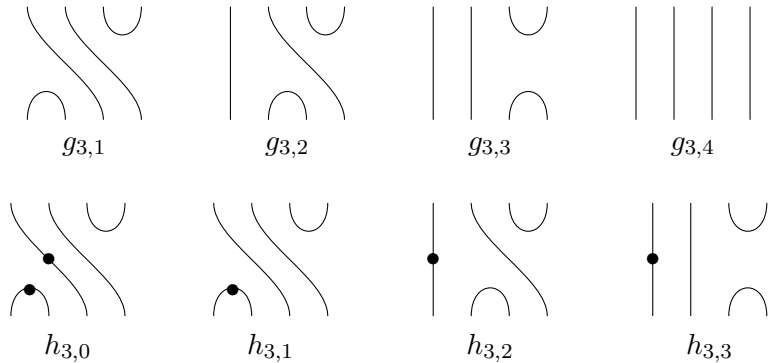

\tikzpic{-0.5}{[scale=0.5]
\draw(0,0)..controls(0,1)and(1,1)..(1,0);
\draw(0,3)..controls(0,2)and(2,1)..(2,0);
\draw(1,3)..controls(1,2)and(3,1)..(3,0);
\draw(2,3)..controls(2,2)and(3,2)..(3,3);
\draw(1.5,-0.2)node[anchor=north]{$g_{3,1}$};
}\qquad
\tikzpic{-0.5}{[scale=0.5]
\draw(1,0)..controls(1,1)and(2,1)..(2,0);
\draw(0,3)--(0,0);
\draw(1,3)..controls(1,2)and(3,1)..(3,0);
\draw(2,3)..controls(2,2)and(3,2)..(3,3);
\draw(1.5,-0.2)node[anchor=north]{$g_{3,2}$};
}\qquad
\tikzpic{-0.5}{[scale=0.5]
\draw(2,0)..controls(2,1)and(3,1)..(3,0);
\draw(0,3)--(0,0);
\draw(1,3)--(1,0);
\draw(2,3)..controls(2,2)and(3,2)..(3,3);
\draw(1.5,-0.2)node[anchor=north]{$g_{3,3}$};
}\qquad
\tikzpic{-0.5}{[scale=0.5]
\draw(2,0)--(2,3);
\draw(0,3)--(0,0);
\draw(1,3)--(1,0);
\draw(3,3)--(3,0);
\draw(1.5,-0.2)node[anchor=north]{$g_{3,4}$};
} \\[12pt]
\tikzpic{-0.5}{[scale=0.5]
\draw(0,0)..controls(0,1)and(1,1)..(1,0);
\draw(0,3)..controls(0,2)and(2,1)..(2,0);
\draw(1,3)..controls(1,2)and(3,1)..(3,0);
\draw(2,3)..controls(2,2)and(3,2)..(3,3);
\draw(1.5,-0.2)node[anchor=north]{$h_{3,0}$};
\draw(0.5,0.68)node{$\bullet$}(1,1.5)node{$\bullet$};
}\qquad
\tikzpic{-0.5}{[scale=0.5]
\draw(0,0)..controls(0,1)and(1,1)..(1,0);
\draw(0,3)..controls(0,2)and(2,1)..(2,0);
\draw(1,3)..controls(1,2)and(3,1)..(3,0);
\draw(2,3)..controls(2,2)and(3,2)..(3,3);
\draw(1.5,-0.2)node[anchor=north]{$h_{3,1}$};
\draw(0.5,0.68)node{$\bullet$};
}\qquad
\tikzpic{-0.5}{[scale=0.5]
\draw(1,0)..controls(1,1)and(2,1)..(2,0);
\draw(0,3)--(0,0);
\draw(1,3)..controls(1,2)and(3,1)..(3,0);
\draw(2,3)..controls(2,2)and(3,2)..(3,3);
\draw(1.5,-0.2)node[anchor=north]{$h_{3,2}$};
\draw(0,1.5)node{$\bullet$};
}\qquad
\tikzpic{-0.5}{[scale=0.5]
\draw(2,0)..controls(2,1)and(3,1)..(3,0);
\draw(0,3)--(0,0);
\draw(1,3)--(1,0);
\draw(2,3)..controls(2,2)and(3,2)..(3,3);
\draw(1.5,-0.2)node[anchor=north]{$h_{3,3}$};
\draw(0,1.5)node{$\bullet$};
}
\caption{The graphical representation of $g_{3,i}$ and $h_{3,j}$.}
\label{fig:gh}
\end{figure}

We recall the recurrence relation via diagrams for $\mathrm{TL}^{A}_{n+1}$.
In \cite{Mor17} (see also \cite{FK97}), Morrison gave a recursive formula for the coefficients of 
the elements, equivalently for $n+1$-strand Temperley--Lieb diagrams.  

By definition, note that the vertical strand which connects the right-most points on the top and on the bottom
is called an inner-most cap.
Any $n+1$-strand Temperley--Lieb diagram has at least one inner-most cap.
This property allows us to have a recurrence relation via diagrams. 

Let $D$ be an $n+1$-strand Temperley--Lieb diagram.
Suppose that an inner-most cap $c$ in $D$ connects the $i$-th bottom point from left and 
$i+1$-th bottom point from left for $1\le i\le n$.
Then, the position of $c$ is defined as $i$.
If the right-most points on the top and on the bottom in $D$ are connected by a vertical strand,
the position of this inner-most cap is defined as $n+1$.

We first recall the recurrence relation the coefficients of the projection for 
type $A$.
\begin{prop}[Proposition 4.1 in \cite{Mor17}]
\label{prop:coefDA}
Let $D$ be an $n+1$-strand Temperley--Lieb diagram.
Let $I(D)\subseteq[1,n+1]$ be the set of positions of inner-most caps in $D$, and 
$D_{i}$ the $n$-strand Temperley--Lieb diagram obtained from $D$ by removing the $i$-th inner-most cap.
The coefficient $\mathrm{coef}^{A}_{n}(D)$ of $D$ in $P_{n}$ satisfies the 
following recurrence relation:
\begin{align}
\label{eq:coefDA}
\mathrm{coef}^{A}_{n}(D)=\sum_{i\in I(D)}\genfrac{}{}{}{}{[i]}{[n+1]}\mathrm{coef}^{A}_{n-1}(D_i).
\end{align}
\end{prop}

In what follows, we will derive the recurrence relation via diagrams 
for $\mathrm{TL}^{D}_{n+1}$.
Let $D$ be an $n+1$-strand Tempereley--Lieb diagram possibly with dots.
For convenience, we simply write the coefficient $\mathrm{coef}_{n+1}(D)$ of $D$ in $Q_{n+1}$
as $D$.

First, we consider the case where $D$ contains an even number of dots.
We define two sets $I^{\bullet}(D)$ and $I^{\circ}(D)$.
We define the set $I^{\bullet}(D):=\{1\}$ if there exists a dotted 
inner-most cap at the position one and $I^{\bullet}(D)=\emptyset$ otherwise.
Similarly, the set $I^{\circ}(D)$ is defined as the set of the positions of inner-most caps
without a dot.
Then, we define $I(D):=I^{\bullet}(D)\cup I^{\circ}(D)$.
\begin{prop}
\label{prop:Deven}
Suppose that $D$ has a diagram consisting of $n+1$ strands with an even number of dots.
Let $I(D)$ be the set as above. We denote by $D_{i}$ the $n$-strand diagram obtained from 
$D$ by removing the $i$-th inner-most cap. Further, if $i=1\in I^{\bullet}(D)$, we add a dot on a cap or a vertical
strand which contains the left-most point at the bottom in $D_{i}$.
Then, we have 
\begin{align}
\label{eq:Dev}
D=\sum_{i\in I(D)}\mathrm{coef}(g_{n,i})D_{i}.
\end{align}
We can recursively apply Eq. (\ref{eq:Dev}) to the diagram $D_{i}$ since $D_{i}$ has 
an even number of dots.
\end{prop}
\begin{remark}
Four remarks are in order:
\begin{enumerate}
\item The diagram $D$ has an even number of dots. Similarly, a diagram $D_{i}$ also has 
an even number of dots. This is because we add a dot to a diagram after the removal of an inner-most cap with a dot 
at position one.  
\item Suppose we have two dots on a cap or a vertical strand after we add a dot on the cap or the vertical strand.
Then, we delete these two dots.
\item If an inner-most cap at position $i\neq1$ has a dot, this inner-most cap 
cannot be removed to produce a diagram of smaller size. 
This comes from the fact that we have no generator which has the same connectivity as $E_{i}$, $i\ge2$, 
and has a dotted cap and a dotted cup.
This is the reason why we have $I^{\bullet}(D)=\{1\}$ if there is a dotted cap at position one,
and $I^{\bullet}(D)=\emptyset$ otherwise.
\item Suppose that $1\in I(D)$. Then, we have an inner-most cap with or without a dot
at the position one. Since we have $\mathrm{coef}(h_{n,0})=\mathrm{coef}(g_{n,1})$ by 
definition, we do not need to distinguish these two coefficients.
\end{enumerate}
\end{remark}

\begin{proof}[Proof of Proposition \ref{prop:Deven}]
Recall we have the recurrence relation (\ref{eq:Q2}) in Proposition \ref{prop:Q1}.
Let $D$ be a diagram with $n+1$ strands and with an even number of dots.
This implies that the algebraic representation of $D$ does not contain the product
$E_0E_1$. Recall that the elements $h_{n,j}$ with $j\ge 1$ contain the product 
$E_0E_1$. If we multiply $h_{n,j}$ by any $g_{m,i'}$ or $h_{m,j'}$ with $m\le n-1$,
the result also contains the product of $E_0E_1$.
This fact and Eq. (\ref{eq:Q2}) imply that $D$ is expressed as a product $D'g_{n,i}$ or $D''h_{n,0}$ such that
$D'$ or $D''$ is a diagram of $n$ strands and $1\le i\le n+1$.
Here, $D'$ or $D''$ is a diagram with $n$ strands and when we consider the product with $g_{n,i}$ or $h_{n,0}$, we
add a vertical strand to the right of $D'$ or $D''$ such that $D'$ or $D''$ and the vertical strand form a diagram
with $n+1$ strands.
By the same argument as the case of type $A$ (see \cite{Mor17} for details), the connectivity of a diagram $D'$ is the same 
as that of $D_{i}$. 
In the case $D=D'g_{n,i}$, it is obvious that the position of dots in $D'$ is the same as $D_{i}$. From these,
we have $D'=D_{i}$.
If $D=D''h_{n,0}$, we have to add a dot on the cap or the vertical strand which contains the left-most bottom point
after deleting the left-most cap, since the graphical representation of $h_{n,0}$ has a vertical strand with a dot 
which contains the left-most top point. 
Since each cap or a vertical strand has at most one dot, if there exist two dots on a cap or a vertical strand, then 
we delete these two dots.
From these, we have $D''=D_{1}$ if $D=D''h_{n,0}$ and $1\in I^{\bullet}(D)$.
Finally, we have $\mathrm{coef}(g_{n,1})=\mathrm{coef}(h_{n,0})$.  
By combining these observations together, we have Eq. (\ref{eq:Dev}).
\end{proof}

\begin{example}
We consider the coefficient of $E_{1}E_{2}E_{3}E_{0}$ in $P_{3}$.
We calculate
\begin{align*}
\tikzpic{-0.5}{[scale=0.6]
\draw(0,0)..controls(0,1)and(1,1)..(1,0)(0,2)..controls(0,1)and(1,1)..(1,2);
\draw(0.5,0.72)node{$\bullet$};
\draw(2,0)..controls(2,1)and(3,1)..(3,0)(2,2)..controls(2,1)and(3,1)..(3,2);
\draw(2.5,1.25)node{$\bullet$};
}&=\genfrac{}{}{}{}{[3]}{[6]}\cdot 
\tikzpic{-0.5}{[scale=0.6]
\draw(0,0)..controls(0,1)and(1,1)..(1,0)(0,2)..controls(0,1)and(1,1)..(1,2);
\draw(2,0)--(2,2);
\draw(2,1)node{$\bullet$};
\draw(0.5,0.72)node{$\bullet$};
}
+
\genfrac{}{}{}{}{[3][4]}{[6][2]}\cdot 
\tikzpic{-0.5}{[scale=0.6]
\draw(0,0)..controls(0,1)and(1,1)..(1,0)(0,2)..controls(0,1)and(1,1)..(1,2);
\draw(2,0)--(2,2);
\draw(2,1)node{$\bullet$};
\draw(0.5,0.72)node{$\bullet$};
}, \\
&=
\left(\genfrac{}{}{}{}{[3]}{[6]}+\genfrac{}{}{}{}{[3][4]}{[6][2]}\right)\genfrac{}{}{}{}{[2]}{[4]}\cdot
\tikzpic{-0.5}{[scale=0.6]
\draw(0,0)..controls(0,1)and(1,1)..(1,0)(0,2)..controls(0,1)and(1,1)..(1,2);
},\\
&=\genfrac{}{}{}{}{[3]^2[2]}{[6][4]}\cdot \genfrac{}{}{}{}{1}{[2]}\cdot
\tikzpic{-0.5}{[scale=0.6]
\draw(0,0)--(0,2);
}, \\
&=\genfrac{}{}{}{}{[3]^2}{[6][4]}.
\end{align*}
Note that we are not allowed to remove the vertical strand with a dot at the 
position $3$ (in the first row).
In the calculation of the coefficient of $E_1E_2E_3E_0$, we have only diagrams which have an even number of dots.

Compare the result with Example \ref{ex:E0321}.
\end{example}

Secondly, we consider the case where a diagram $D$ contains a single dot, which implies that 
$D$ corresponds to an element containing the product $E_0E_1$ in its algebraic representation. 
The connectivity of the diagram $D$ is different from $\mathbf{1}$ by the property (P2c).
Let $I(D)\subseteq[1,n+1]$ be the set of positions of inner-most caps possibly with a dot.
If $n+1\in I(D)$, this means that a diagram has an inner-most cap as a vertical 
strand connecting the right-most marked points on the top and bottom. 
Given $i\in I(D)$, we denote by $D_{i}$ the diagram obtained from $D$ by deleting 
the inner-most cap at the position $i$.
The size of $D_{i}$ is one smaller than that of $D$.
We have two cases: 1) $D_{i}\neq\mathbf{1}$, and 2) $D_i=\mathbf{1}$.

Case 1). We define $\mathcal{D}^{e}_{i}$ as the set of diagrams $D$ such that $D$ has the 
same connectivity of caps, cups and vertical strands as $D_{i}$, and has 
an even number of dots. The positions of dots in the diagram $D'\in\mathcal{D}^{e}_{i}$ 
satisfy the property (P2a).
Similarly, we denote by $D_{i}^{\mathrm{odd}}$ the unique diagram which has the same connectivity 
of caps, cups and vertical strands as $D_{i}$, and has a unique dot satisfying 
the property (P2b).

Case 2). Since $D_i=\mathbf{1}$, we define $\mathcal{D}^{e}_{i}:=\{\mathbf{1}\}$ and 
$D_{i}^{\mathrm{odd}}:=\emptyset$ by the property (P2c).

\begin{prop}
\label{prop:Dodd}
Let $D$ be a diagram with a unique dot, 
and $I(D)$, $\mathcal{D}^{e}_{i}$, and $D_{i}^{\mathrm{odd}}$ be as above.
Then, we have 
\begin{align}
\label{eq:Dodd}
D=\sum_{i\in I(D)}\mathrm{coef}(h_{n,i})\left(\sum_{D'\in\mathcal{D}^{e}_{i}}D'-[2]D_{i}^{\mathrm{odd}}\right)
+\sum_{i\in I(D)}2^{\delta(i\equiv1)}\mathrm{coef}(g_{n,i})D_{i}^{\mathrm{odd}},
\end{align}
where $\delta(P)$ is the Kronecker delta function, namely, $\delta(P)=1$ if $P$ is true and $\delta(P)=0$
otherwise.
\end{prop}
	
\begin{remark}
\label{remark:Dodd}
Two remarks are in order:
\begin{enumerate}
\item The factor $2^{\delta(i\equiv1)}$ comes from the fact that we have $\mathrm{coef}(h_{n+1,0})=\mathrm{coef}(g_{n+1,1})$,
and $I(D)\subseteq[1,n+1]$.
The set $I(D)$ of the positive integers cannot detect the generator $h_{n+1,0}$ since $0\notin I(D)$.
\item The factor $-[2]$ in front of $D_{i}^{\mathrm{odd}}$ comes from the fact that if we multiply $h_{n,i}$
with $h_{m,j}$ such that $i,j\neq0$, then this product gives $-[2]h'$ with some $h'\in\mathrm{TL}^{D}_{n}$.
In other words, both $h_{n,i}$ and $h_{m,j}$ contain the product $E_0E_1$, and we use
$E_{1}^2=-[2]E_1$ or $E_0^2=-[2]E_0$ in the calculation of the product.
\end{enumerate}
\end{remark}

\begin{proof}[Proof of Proposition \ref{prop:Dodd}]
We rewrite the recurrence relation (\ref{eq:Q2}) in Proposition \ref{prop:Q1} 
in terms of diagrams.
Recall that the diagram $D$ has a unique single dot which implies that the element 
corresponding to $D$ contains the product $E_0E_1$.

If we have $D=D'g_{n,i}$ or $D=D''h_{n,0}$ where $D'$ or $D''$ is an element in $Q_{n-1}$,  $D'$ or $D''$ has to contain 
the product $E_{0}E_{1}$ since $g_{n,i}$ or $h_{n,0}$ does not contain 
the product $E_{0}E_{1}$.
Further we have $\mathrm{coef}(g_{n,1})=\mathrm{coef}(h_{n,0})$.
This gives the factor $2^{\delta(i\equiv1)}$.
By the same argument as the proof of Proposition \ref{prop:Deven},
$D'$ or $D''$ has the same connectivity as $D_{i}^{\mathrm{odd}}$. 
The number of dots in $D'$ or $D''$ is one since it contains $E_{0}E_{1}$.
Then, the diagram $D'$ or $D''$ satisfies the property (P2b).
From these, we have the second sum in the right hand side of Eq. (\ref{eq:Dodd}).

Suppose that we have $D=D'h_{n,i}$ with $1\le i\le n$ where $D'$ is an element in $Q_{n-1}$.
Then, since $h_{n,i}$ contains the product $E_{0}E_{1}$, 
we have two cases for $D'$: a) the diagram $D'$ has an even number of dots, and 
b) $D'$ has a unique single dot.
\begin{enumerate}[(a)]
\item Again, by the same argument as the proof of Proposition \ref{prop:Deven}, 
the connectivity of $D'$ is the same as diagrams in $\mathcal{D}^{e}_{i}$.
Therefore, it is enough to show that a diagram $D'\in\mathcal{D}^{i}_{e}$ satisfies 
$D=D'h_{n,i}$. 
However, the facts that $D$ and $h_{n,i}$ contain a single dot, and $D'$ contains 
an even number of dots, imply that if $D'\in\mathcal{D}^{i}_{e}$, then $D=D'h_{n,i}$.
From this observation, we have the term $\sum_{i\in I(D)}\mathrm{coef}(h_{n,i})\sum_{D'\in\mathcal{D}^{e}_{i}}D'$.
\item Since $D'$ has a unique single dot, the diagram is uniquely fixed if the connectivity 
of caps, cups and vertical strands is fixed. The condition that $D=D'h_{n,i}$ implies that
the connectivity of $D'$ is the same as that of $D_{i}^{\mathrm{odd}}$. Then, it is easy to see 
that $D'=D_{i}^{\mathrm{odd}}$ since $D'$ contains a unique dot.
The element corresponding to $D_{i}^{\mathrm{odd}}$ contains the product $E_{0}E_{1}$.
By the second remark in Remark \ref{remark:Dodd}, $D_{i}^{\mathrm{odd}}h_{n,i}$ yields 
the factor $-[2]$.
By combining these observations, we have the term 
$-[2]\sum_{i\in I(D)}\mathrm{coef}(h_{n,i})D_{i}^{\mathrm{odd}}$.
\end{enumerate}
From the two cases (a) and (b), we have the first sum in the right hand side of Eq. (\ref{eq:Dodd}).
This completes the proof.
\end{proof}

\subsection{Relation between types \texorpdfstring{$A$}{A} and \texorpdfstring{$D$}{D}}
\label{sec:relAD}
In this section, we will show that a linear combination of the diagrams in $\mathrm{TL}^{D}_{n+1}$
satisfies the recurrence relation of a diagram in $\mathrm{TL}^{A}_{n+1}$.
This linear combination of the diagrams of type $D$ plays a central role when we 
interpret the recursive relation (\ref{eq:Dodd}) in terms of the Dyck tilings.

Let $D$ be an $n+1$-strand Temperley--Lieb diagram for $\mathrm{TL}^{D}_{n+1}$.
The diagram may have dots.
We denote by $D^{A}$ the diagram obtained from $D$ by deleting all dots.
In other words, we focus on the only connectivity of $D$ and ignore the dots.
This operation naturally defines the forgetful map $\mathtt{For}:D\mapsto D^{A}$.
Let $\mathcal{D}^{e}(D^A)$ be the set of diagrams $D$ such that $\mathtt{For}(D)=D^{A}$
and they have an even number of dots satisfying (P2a).
We have a unique diagram which has an odd number of dots and has the same connectivity as $D^{A}$ 
if $D^{A}\neq\mathbf{1}$. 
Therefore, we denote by $D^{\mathrm{odd}}$ this unique diagram. As a consequence, the diagram $D^{\mathrm{odd}}$ 
has a single dot, and it satisfies (P2b).
Here, the superscript $A$ in $D^A$ stands for type $A$, and the superscript $\mathrm{odd}$ in $D^{\mathrm{odd}}$
stands for a diagram of type $D$ with a single dot.
\begin{defn}
Let $D^{A}$, $\mathcal{D}^{e}(D^A)$ and $D^{\mathrm{odd}}$ be the diagram given as above.
We define the following linear combination:
\begin{align}
\label{eq:defDAD}
D^{A}:=\sum_{D\in\mathcal{D}^{e}(D^A)}D-[2]D^{\mathrm{odd}}.
\end{align}
\end{defn}

The linear combination $D^{A}$ in Eq. (\ref{eq:defDAD}) can be regarded as 
an element of Temperely--Lieb algebra of type $A$ by the next proposition.
\begin{prop}
\label{prop:DAD}
The linear combination $D^{A}$ satisfies the recurrence relation (\ref{eq:coefDA}) of type $A$
in Proposition \ref{prop:coefDA}.
\end{prop}
\begin{proof}
If $D^{A}=\mathbf{1}$, we have $\mathcal{D}^{e}(D^{A})=\{\mathbf{1}\}$ and $D^{\mathrm{odd}}=\emptyset$.
In this case, the claim is trivial.
For $D^{A}\neq\mathbf{1}$, it is enough to show that the linear combination (\ref{eq:defDAD}) 
satisfies the recurrence relation (\ref{eq:coefDA}) of type $A$ by use of
the recurrence relations (\ref{eq:Dev}) and (\ref{eq:Dodd}).  
We compute the coefficient of $D^{A}_{i}$, $i\in I(D^{A})$, where $D^{A}_{i}$ 
is the diagram of type $A$ obtained from $D^{A}$ by removing the inner-most
cap at position $i$. 
We consider the three cases: the first one is $i=n+1$, the second one is $1\le i\le n$,
and the third one is $i=1$.

Firstly, suppose that $n+1\in I(D^A)$.
The diagram $D^A$ has a vertical strand which connects the two right-most 
marked points on the top and bottom. This vertical strand is inner-most by definition.
We calculate the contributions of the right hand side of Eq. (\ref{eq:defDAD})
to $D^{A}_{n+1}$.
When $D\in\mathcal{D}^{e}(D^{A})$, and the $n+1$-st inner-most cap does not 
have a dot, we have a coefficient $\mathrm{coef}(g_{n+1,n+1})=1$. 
Further, the new diagram (obtained from $D$ by removing the right-most vertical
strand) contains an even number of dots. 
We have all diagrams with an even number of dots and with the same connectivity as $D^{A}_{n+1}$.
If $D$ has a dot at the $n+1$-st inner-most
cap, this does not contribute to $D^{A}_{n+1}$ since we have a recurrence relation
(\ref{eq:Dev}).
We apply Eq. (\ref{eq:Dodd}) to the diagram $D^{\mathrm{odd}}$.
Note that the coefficient of $h_{n,n+1}$ is equal to zero since it is not defined. 
For the second term, we have $\mathrm{coef}(g_{n,n+1})D^{\mathrm{odd}}_{n+1}$.
By combining these with Eq. (\ref{eq:defDAD}), 
we have the linear combination of the diagrams for $D^{A}_{n+1}$ with the coefficient $1$.

Secondly, suppose that $1<i\le n$ and $i\in I(D^{A})$.
We apply the recurrence relations (\ref{eq:Dev}) and (\ref{eq:Dodd}) to Eq. (\ref{eq:defDAD}).
Explicitly, we have
\begin{align*}
&\mathrm{coef}(g_{n,i})\sum_{D\in \mathcal{D}^{e}(D^{A}_{i})}D
-[2]\left(\mathrm{coef}(h_{n,i})\sum_{D\in \mathcal{D}^{e}(D^{A}_{i})}D-[2]\mathrm{coef}(h_{n,i})D^{\mathrm{odd}}_{i}
+\mathrm{coef}(g_{n,i})D^{\mathrm{odd}}_{i}\right), \\
&\qquad=\genfrac{}{}{}{}{[i]}{[n+1]}\left(\sum_{D\in \mathcal{D}^{e}(D^{A}_{i})}D-[2]D^{\mathrm{odd}}_{i}\right), \\
&\qquad=\genfrac{}{}{}{}{[i]}{[n+1]}D^{A}_{i},
\end{align*}
where $D\in\mathcal{D}^{e}(D^{A}_{i})$ (resp. $D^{\mathrm{odd}}_{i}$) are the diagrams of even (resp. odd) number of dots 
such that $\mathtt{For}(D)=D^{A}_{i}$ (resp. $\mathtt{For}(D^{\mathrm{odd}}_{i})=D^{A}_{i}$).

Finally, suppose that $i=1\in I(D^{A})$.
By applying the recurrence relations (\ref{eq:Dev}) and (\ref{eq:Dodd}) to Eq. (\ref{eq:defDAD}),
we obtain 
\begin{align}
\label{eq:DA2}
\begin{aligned}
&2\cdot \mathrm{coef}(g_{n,1})\sum_{D\in \mathcal{D}^{e}}D
-[2]\left(\mathrm{coef}(h_{n,1})\sum_{D\in \mathcal{D}^{e}}D-[2]\mathrm{coef}(h_{n,1})D^{\mathrm{odd}}_{1}
+2\cdot\mathrm{coef}(g_{n,1})D^{\mathrm{odd}}_{1}\right), \\
&\qquad=\genfrac{}{}{}{}{1}{[n+1]}\left(\sum_{D\in \mathcal{D}^{e}}D-[2]D^{\mathrm{odd}}_{1}\right), \\
&\qquad=\genfrac{}{}{}{}{1}{[n+1]}D^{A}_{1}.
\end{aligned}
\end{align}
where $\mathcal{D}^{e}:=\mathcal{D}^{e}(D^{A}_{1})$.
The first factor $2$ in Eq. (\ref{eq:DA2}) comes from the fact that we have two generators 
$E_{1}E_2E_{0}$ and $E_{0}E_2E_{1}$, whose diagrams are the same as $E_1$ and $E_0$ 
if we ignore the dots. 

Form these, the linear combination (\ref{eq:defDAD}) satisfies the recurrence relation (\ref{eq:coefDA}).
\end{proof}

\begin{example}
Let $D^A$ be a diagram without dots:
\begin{align*}
D^{A}=
\tikzpic{-0.5}{[xscale=0.5,yscale=0.5]
\draw(0,0)..controls(0,1)and(1,1)..(1,0);
\draw(0,2)..controls(0,1)and(1,1)..(1,2);
\draw(2,0)--(2,2);
}.
\end{align*}
Then, we consider the following linear combination in $\mathrm{TL}^{D}_{3}$:  
\begin{align}
\label{eq:E0E1}
D^{A}=E_{1}+E_{0}+E_{1}E_{2}E_{0}+E_{0}E_{2}E_{1}-[2]E_0E_1.
\end{align}
Here, we make use of the algebraic representations of the diagrams.
By a simple calculation using Eq. (\ref{eq:coefDA}), we have 
\begin{align*}
\mathrm{coef}^{A}(D^A)=\genfrac{}{}{}{}{[2]}{[3]},
\end{align*}
where $\mathrm{coef}^{A}(D)$ is the coefficient of the diagram $D$ in the type $A$ case.
By a similar calculation, the right hand side of Eq. (\ref{eq:E0E1}) gives 
\begin{align*}
\genfrac{}{}{}{}{[3]}{[4]}+\genfrac{}{}{}{}{[3]}{[4]}
+\genfrac{}{}{}{}{1}{[4]}+\genfrac{}{}{}{}{1}{[4]}
-[2]\genfrac{}{}{}{}{[2]^3}{[4][3]}=\genfrac{}{}{}{}{[2]}{[3]}.
\end{align*}
\end{example}

\subsection{Recurrence relation via diagrams with a unique dot}
By combining the result in Section \ref{sec:rrd} with that 
in Section \ref{sec:relAD},  we rewrite the recurrence relation 
(\ref{eq:Dodd}) in terms of the diagrams of types $A$ and $D$.

Let $D$ be the diagram with a unique dot.
Let $D_{i}$ be the diagram obtained from $D$ by removing an 
inner-most cap at position $i$.
We define $D^{A}_{i}:=\mathtt{For}(D_{i})$ where $\mathtt{For}$ is the forgetful map.
Similarly, $D_{i}^{\mathrm{odd}}$ is a diagram with a unique dot such that 
$\mathtt{For}(D_{i}^{\mathrm{odd}})=D_{i}$.

The following is a direct consequence of Propositions \ref{prop:Dodd} and \ref{prop:DAD}.

\begin{prop}
\label{prop:DinDAD}
Let $D$, $D^{A}_{i}$ and $D_{i}^{\mathrm{odd}}$ be diagrams as above.
Then, we have
\begin{align}
\label{eq:DinDAD}
D=\sum_{i\in I(D)}\mathrm{coef}(h_{n,i})D^{A}_{i}
+\sum_{i\in I(D)}2^{\delta(i\equiv1)}\mathrm{coef}(g_{n,i})D_{i}^{\mathrm{odd}}.
\end{align}
\end{prop}

Note that the recurrence relation (\ref{eq:Dodd}) is not minus sign free, 
{\it i.e.}, we have a term with minus sign $-[2]D_{i}^{\mathrm{odd}}$.
However, the recurrence relation (\ref{eq:DinDAD}) is a minus sign free
because of Proposition \ref{prop:DAD}.
The coefficient of $D$ in $Q_{n}$ is expressed as the linear combination 
of $D^{A}_{i}$ and $D_{i}^{\mathrm{odd}}$ with positive coefficients.
However, the recurrence relation contains  the diagrams not only of type $D$ 
but also of type $A$.
This form plays a central role when we interpret the recurrence relation 
(\ref{eq:DinDAD}) in terms of the Dyck tilings in the next section.

\section{Dyck tilings}
\label{sec:DT}
\subsection{Definition}
We briefly introduce the notion of cover-inclusive Dyck tilings 
following \cite{KMPW12,SZJ12}.

A {\it Dyck path} of size $n$ is a lattice path from $(0,0)$ to $(2n,0)$ 
such that it consists of up steps $(1,1)$ and down steps $(1,-1)$ and 
it never goes below the horizontal line $y=0$.
By its definition, a Dyck path has $n$ up steps and $n$ down steps.
For example, we have five Dyck paths of size three:
\begin{align*}
UUUDDD \qquad UUDUDD \qquad UUDDUD \qquad UDUUDD \qquad UDUDUD
\end{align*}
where $U$ (resp. $D$) stands for an up (resp. down) step.
We call the Dyck path $U^{n}D^{n}$ the top Dyck path of size $n$.
Let $\lambda,\mu$ be two Dyck paths.
We write $\lambda\le\mu$ if $\mu$ is above $\lambda$.

A {\it ribbon} is a connected skew shape which does not contain
a $2$-by-$2$ rectangle.
A ribbon is called a {\it Dyck tile} if the centers of the unit cells 
in the ribbon form a Dyck path.
A unit cell is a Dyck tile of size zero, and the size of a Dyck tile $d$ 
is the size of the Dyck path which characterizes $d$.
We call a Dyck tile of size zero a trivial Dyck tile, and other Dyck tiles
non-trivial Dyck tiles.

Let $\lambda,\mu$ be two Dyck paths such that $\lambda\le\mu$.
We consider the tiling of the region above $\lambda$ and below $\mu$
by Dyck tiles.
In this paper, we consider only cover-inclusive Dyck tilings.
Here, ``cover-inclusiveness" means that if we move a Dyck tile 
downward by $(0,-2)$, then it is contained in another Dyck tile
or strictly below the path $\lambda$.
\begin{figure}[ht]
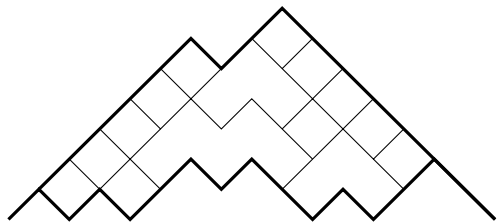

\tikzpic{-0.5}{[scale=.4]
\draw[very thick](0,0)--(1,1)--(2,0)--(3,1)--(4,0)--(6,2)--(7,1)--(8,2)--(10,0)--(11,1)--(12,0)--(14,2)--(16,0);
\draw[very thick](1,1)--(6,6)--(7,5)--(9,7)--(16,0);
\draw(2,2)--(3,1)--(7,5)(3,3)--(4,2)(4,4)--(5,3)(5,5)--(7,3)--(8,4)--(10,2)(4,2)--(5,1);
\draw(8,6)--(13,1)(10,6)--(9,5)(11,5)--(9,3)(12,4)--(9,1)(13,3)--(12,2);
}
\caption{An example of a cover-inclusive Dyck tiling of size $8$}
\label{fig:DT}
\end{figure}
We give an example of a cover-inclusive Dyck tiling in Figure \ref{fig:DT}.
We have fifteen Dyck tiles, and three non-trivial Dyck tiles.

\begin{remark}
The cover-inclusive Dyck tilings correspond to Dyck tilings
which satisfy the rule I in \cite{SZJ12}.
\end{remark}

A Dyck tile $d$ consists of several unit squares. Given a unit box $s$, we denote 
by $h(s)$ the height of the center of $s$.
\begin{defn}
We define the height of $d$ by
\begin{align*}
h(d):=\min\{h(s):s\in d\}.
\end{align*}
\end{defn}

In Figure \ref{fig:DT}, the heights of three non-trivial Dyck tiles are 
one, two, and four.

\subsection{Correspondence between a Dyck path and a Temperley--Lieb diagram}
\label{sec:DpTL}
In this section, we give a correspondence between a Dyck path of size $n$ and 
an $n$-strand Temperley--Lieb diagram $D$ possibly with dots.
We fold an $n$-strand Temperley--Lieb diagram $D$ down to the right such that $2n$ points are in line.
Then, we obtain a graph $\mu$ with $2n$ points such that a point is connected 
to another point by a cap. We have $n$ caps in the graph $\mu$.
Note that the strands in $D$ are non-intersecting,
therefore the caps in $\mu$ are also non-intersecting.
Obviously, we have a one-to-one correspondence between a strand in $D$ and 
a cap in $\mu$.
A cap in $\mu$ has a dot if the corresponding strand in $D$ has a dot.
If we replace the left end point of a cap with $U$, and the right end point of 
a cap with $D$, then we have a Dyck path.
The Dyck path which consists of possibly dotted caps corresponds to 
the diagram $D$.

In Figure \ref{fig:DyckTLd}, we give the Dyck path with a dot which 
corresponds to the diagram $A$ in Example \ref{ex:AB}.
\begin{figure}[ht]
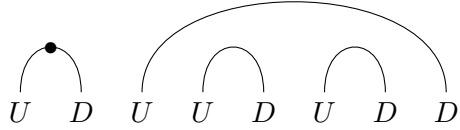

\tikzpic{-0.5}{[xscale=.8,yscale=.8]
\draw(0,0)node[anchor=north]{$U$}..controls(0,1)and(1,1)..(1,0)node[anchor=north]{$D$};
\draw(0.5,0.72)node{$\bullet$};
\draw(2,0)node[anchor=north]{$U$}..controls(2,2)and(7,2)..(7,0)node[anchor=north]{$D$};
\draw(3,0)node[anchor=north]{$U$}..controls(3,1)and(4,1)..(4,0)node[anchor=north]{$D$};
\draw(5,0)node[anchor=north]{$U$}..controls(5,1)and(6,1)..(6,0)node[anchor=north]{$D$};
}
\caption{A Dyck path corresponding to a $4$-strand Temperley--Lieb diagram.}
\label{fig:DyckTLd}
\end{figure}
The corresponding Dyck path is $UDUUDUDD$ and the left-most cap has a dot.

We call a graph consisting of $2n$ points and $n$ non-crossing caps with or without dots
a {\it chord diagram}.
The number of chord diagrams consisting of $n$ caps without dots is 
given by the Catalan number $C_{n}$.
Similarly, the number of chord diagrams of $n$ caps with dots which
correspond to the $n$ strand Temperley--Lieb diagrams of type $D$ 
is given by $\dim\mathrm{TL}_{n}^{D}$.
This is easily obtained since the number of chord diagrams with an even number of dots
is $\genfrac{}{}{}{}{1}{2}\genfrac{[}{]}{0pt}{}{2n}{n}$, and those with an odd number of dots
is $C_{n}-1$. Here, we use the property (P2c), which implies that the chord diagram 
corresponding to the identity $\mathbf{1}$ cannot have a dot.
As a consequence, the number of chord diagrams with dots is given by 
\begin{align*}
\genfrac{}{}{}{}{1}{2}\genfrac{[}{]}{0pt}{}{2n}{n}+C_n-1=\dim\mathrm{TL}_{n}^{D}.
\end{align*}

\subsection{Hermite histories}
We introduce the notion of Hermite histories on a cover-inclusive Dyck tiling: 
horizontal and vertical.
(Horizontal Hermite histories are introduced in \cite[Section 3]{KMPW12}.)
\begin{figure}[ht]
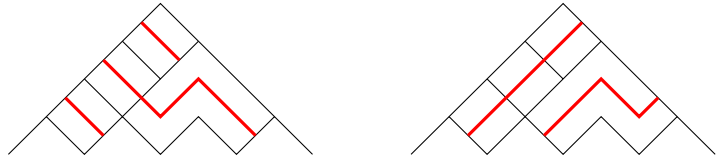

\tikzpic{-0.5}{[scale=0.5]
\draw(0,0)--(1,1)--(2,0)--(3,1)--(4,0)--(5,1)--(6,0)--(7,1)--(8,0);
\draw(1,1)--(4,4)--(7,1);
\draw(2,2)--(3,1)--(5,3)(3,3)--(4,2);
\draw[very thick,red](3.5,3.5)--(4.5,2.5)(2.5,2.5)--(4,1)--(5,2)--(6.5,0.5)(1.5,1.5)--(2.5,0.5);
}
\qquad
\tikzpic{-0.5}{[scale=0.5]
\draw(0,0)--(1,1)--(2,0)--(3,1)--(4,0)--(5,1)--(6,0)--(7,1)--(8,0);
\draw(1,1)--(4,4)--(7,1);
\draw(2,2)--(3,1)--(5,3)(3,3)--(4,2);
\draw[very thick,red](1.5,0.5)--(4.5,3.5)(3.5,0.5)--(5,2)--(6,1)--(6.5,1.5);
}
\caption{Horizontal Hermite history (left) and vertical Hermite history (right).}
\label{fig:Hh}
\end{figure}
We draw a line in Dyck tiles in a Dyck tiling as follows.
The boundary of a Dyck tile consists of up steps ($\diagup$) and down steps ($\diagdown$). 
For a horizontal Hermite history, we write a line from the left-most up step to the right-most
up step in a Dyck tile. We concatenate the lines in Dyck tiles if it is possible to extend.
A line passing through Dyck tiles is called a {\it horizontal trajectory}.
Similarly, for a vertical Hermite history, we write a line from the left-most down step 
to the right-most down step in a Dyck tile. We extend the lines if possible.
A line obtained as above is called a {\it vertical trajectory}.
We emphasize that the length of a trajectory is larger than zero. 
One can consider trajectories of length zero in an Hermite history (see Remark \ref{remark:traj}).
Figure \ref{fig:Hh} gives examples of horizontal and vertical Hermite histories.
We have three horizontal trajectories in the left figure, and two vertical
trajectories in the right figure. These trajectories have length larger than zero.

A horizontal Hermite history  can be obtained as a mirror image of a vertical
Hermite history, and vice versa.

\begin{prop}
\label{prop:traj}
Let $T$ be a Dyck tiling  of size $n$ above $\lambda$ and below $\mu$.
A horizontal or vertical Hermite history $H(T)$ of $T$ has the following properties:
\begin{enumerate}[(a)]
\item $H(T)$ contains at most $n-1$ trajectories with non-zero length.
\item All trajectories are non-crossing.	
\item Any trajectory starts from an up step in $\mu$ for the horizontal Hermite history. 
Similarly, all trajectories ends at a down step in $\mu$ for the vertical Hermite history.
\end{enumerate}
\end{prop}
\begin{proof}
We prove only for a vertical Hermite history since the horizontal case can be proven in a similar way.

(a) By construction of a trajectory, an down step in $\lambda$ is connected to 
an down step in $\mu$ by a trajectory. 
Let $d$ be the right-most down step in $\lambda$.
Since $\lambda$ and $\mu$ share the same down step $d$ at their right ends, there
is no trajectory with non-zero length which starts from $d$.
If we take $\lambda=(UD)^{n}$ and $\mu=U^nD^n$, and $T$ does not have
non-trivial Dyck tiles, then it is obvious that we have $n-1$ trajectories with non-zero length. 
Since we have $n$ down steps in $\lambda$ and $\mu$, the maximal number of 
trajectories with non-zero is $n-1$.

(b) Any Dyck tile $t$ in $T$ has a line inside $t$. The right-most down step 
in $t$ is attached to the left-most down step of another Dyck tile, or to 
a down step in $\mu$. This property insures that all trajectories are non-crossing.

(c) A down step $d$ in $\mu$  is a down step of a Dyck tile $t$, or a down step of $\lambda$. 
If $d$ is the right-most down step of $t$, then there exists a trajectory which ends at $d$.
Otherwise, there is no trajectory which ends at $d$ and has non-zero legnth.
Recall that we consider a cover-inclusive Dyck tiling $T$. 
Especially, if two Dyck tiles $t_1$ is on the top of another Dyck tile $t_2$,
the size of $t_1$ is weakly smaller than that of $t_2$.
This means that the right end of the trajectory containing $t_2$ is right to 
the right end of the trajectory containing $t_1$.
Recall we concatenate lines in Dyck tiles to obtain a trajectory. 
If a trajectory is below another trajectory, the right end of the former 
is right to the latter. 
These implies that the right end of a trajectory ends 
at a down step in $\mu$.
\end{proof}

\begin{remark}
\label{remark:traj}
Suppose that we have $m\le n-2$ trajectories in a vertical Hermite history $H(T)$. 
By the proof of (c) in Proposition \ref{prop:traj}, there exists 
a down step $d$ in $\mu$ such that there exists no trajectory 
ending at $d$. If we regard this situation as a trajectory of length zero,
we have $n$ trajectories in total. Especially, if $n=1$, we have 
no trajectories of length larger than zero. 
This situation can be interpreted as follows: there 
is a trajectory of length zero starting from the unique down step in 
$\lambda$ to the unique down step in $\mu$.
Therefore, if we allow the trajectories of length zero, then
any Hermite history contains $n$ trajectories.	
\end{remark}

\subsection{Dyck tilings and the diagram with an even number of dots}
Let $D$ be an $n+1$-strand Temperley--Lieb diagram with an even number of dots.
Then, by unfolding the diagram $D$, we have a chord diagram $C$ of size $2n+2$.
Since $D$ does not contain the product $E_0E_1$ in the algebraic representation, 
the chord diagram $C$ also has an even number of dots. 
Since $D$ contains only an even number of dots, $D$ satisfies the 
recurrence relation (\ref{eq:Dev}).

First, suppose that the Jones--Wenzl projection $R_{n}$ for a Temperley--Lieb algebra has the form
\begin{align*}
R_{n}=R_{n-1}\left(\sum_{i\in I(D)}\genfrac{}{}{}{}{\langle i\rangle}{\langle n+1\rangle}D_{i}\right),
\end{align*}
where $R_n$ is the projection, $D_{i}$ is the diagram obtained from $D$ by removing the inner-most
cap at position $i$ and $\langle n\rangle$ is a function of $n$.
The projections for types $A$ and $B$, and type $D$ with $E_0E_1=0$ have this form.
In type $A$, we have $\langle n\rangle$ is simply $[n]$, and 
$\langle n\rangle$ is simply $[n]_s$ for type $B$ (see \cite{Shi24} for a definition of $[n]_s$).
In type $D$ (an even number of dots), $\langle n\rangle$ is $\genfrac{}{}{}{}{[2(n-1)]}{[n-1]}$. 
Once again, note that this works well only in the case of the diagrams of type $D$ with an even number of dots.
If the diagram contains a unique dot, then the recurrence relation (\ref{eq:Dodd}) contains 
the coefficients
\begin{align*}
\mathrm{coef}(h_{n,i})=\genfrac{}{}{}{}{[n][n+1-i]}{[2n][n+1]},
\end{align*}
for $1\le i\le n$, and these coefficients cannot be written as $\langle i\rangle/\langle n+1\rangle$.

Secondly, we have the similarity between type $D$ (an even number of dots) and type $B$. 
In type $B$, we cannot remove the inner-most cap at the position $i>1$ if it has 
a dot \cite{Shi24}. Similarly in type $D$, the diagram of $E_0$ has two dots, and we are not allowed to have a diagram
which has the same connectivity as $E_{i}$, $i>1$, and has two dots on a cap and a cup.
Therefore, in type $D$, we cannot remove the inner-most cap with a dot if its position is $i>1$.

The second observation can be easily translated in terms of Dyck tilings.
Some Dyck tilings are not allowed in type $B$, and such Dyck tilings are also not 
allowed in type $D$ as well.

We can apply the same algorithm for type $B$ studied in \cite{Shi24} for these diagrams of type $D$.
When a diagram $D$ contains a cap with a dot, this cap corresponds to a pair 
of an up step and a down step in the chord diagram $C$ corresponding to $D$.
Let $D$ be a Dyck tiling whose bottom Dyck path is the same as $C$, and 
top Dyck path is $U^{n+1}D^{n+1}$.
Let $c$ be a dotted cap in a chord diagram, and $x_{l}$ (resp. $x_{r}$)
the position of the left (resp. right) end of $c$ in the chord diagram $C$.
Then, the size of the cap $c$ is defined to be 
$l(c):=(x_{r}-x_{l}+1)/2$.

\begin{defn}
Let $\mu$ be the Dyck path.
The generating function of Dyck tilings above $\mu$ and below the top path is defined to be
\begin{align*}
Z(\mu):=
{\sum_{T}}'\prod_{d\in T}\genfrac{}{}{}{}{[h(d)][2h(d)-2]}{[2h(d)][h(d)-1]},
\end{align*}
where $d\in T$ is a Dyck tile in a Dyck tiling $T$, we define $\genfrac{}{}{}{}{[2h(d)-2]}{[h(d)-1]}:=1$ if $h(d)=1$, 
and the sum ${\sum}'$ is taken over all Dyck tilings such that there is no Dyck tile
of size $l(c)$ above the dotted cap $c$.
\end{defn}

Then, we have the following theorem.
\begin{theorem}
\label{thrm:Deven}
Let $D$ be a diagram with an even number of dots, and $\mu$ the Dyck path corresponding to $D$.
The coefficient $\mathrm{coef}(D)$ of $D$ in $Q_{n}$ is given by
\begin{align*}
\mathrm{coef}(D)=Z(\mu).
\end{align*}
\end{theorem}
\begin{proof}
We prove that the generating function $Z(\mu)$ satisfies the 
recurrence relation (\ref{eq:Dev}).
Fix a Dyck tiling $D$ above $\mu$ and below the top path $U^{n+1}D^{n+1}$
such that there is no Dyck tile of size $l(c)$ above a dotted cap $c$.
Let $d$ be the left-most down step in the top path, 
and consider a trajectory $T_d$ in the vertical Hermite history of $D$ which 
contains the down step $d$.
We have two cases: 1) $T_d$ connects the down step $d$ and a step in the 
bottom Dyck path $\mu$, and 2) $T_{d}$ does not touch a down step in $\mu$.
\begin{enumerate}
\item The trajectory $T_{d}$ consists of only trivial Dyck tiles. 
Let $h_{low}$ be the height of the lowest unit cell in $T_{d}$.
Then, the product of the weights of the Dyck tiles in $T_{d}$ is given by
\begin{align}
\label{eq:wtTd}
\genfrac{}{}{}{}{[n][2(h_{low}-1)]}{[2n][h_{low}-1]}
\end{align}
if $h_{low}\ge2$ and $[n]/[2n]$ if $h_{low}=1$.
Since $T_d$ touches a step in the bottom Dyck path $\mu$, there is an inner-most cap 
at position $h_{low}$.
To remove the cap at position $h_{low}$ corresponds to give the weight (\ref{eq:wtTd}),
and to delete the trajectory $T_{d}$.
If $h_{low}=1$, we remove the cap at position one. If $h_{low}\ge2$, the cap at position
$h_{low}$ does not have a dot since this cap is not outer-most.
By removing an inner-most cap, we have a Dyck tiling of smaller size. 
\item The trajectory $T_{d}$ consists of only trivial Dyck tiles, however, the low end of $T_d$
touches a down step of non-trivial Dyck tile $d_1$.
As in the case of (1), the product of the weights of the Dyck tiles in $T_d$ is given 
by Eq. (\ref{eq:wtTd}).
By the same argument as (1), to remove the cap at position $h_{low}$ corresponds to 
give the weight, to delete the trajectory $T_{d}$ and to shrink the size of $d_1$ by one.
Suppose that we have a dotted cap $c$ and there is a non-trivial Dyck tile of size $l(c)$
above $c$.
The above observation implies that the trajectory $T_d$ detects the inner-most dotted cap 
at position $h_{low}$. 
However, we are not allowed to remove the inner-most dotted cap at position $h_{low}\ge2$.
This is avoidable if we consider only Dyck tilings such that 
there is no Dyck tile of size $l(c)$ above $c$. In fact, the generating function $Z(\mu)$
satisfies this condition.
\end{enumerate}
Suppose that the bottom path of the Dyck tiling corresponds to the diagram $D$.
By the correspondence between a Temperley--Lieb diagram and a chord diagram,
the Dyck tiling of smaller size corresponds to $D_{i}$ in Eq. (\ref{eq:Dev}).
By combining (1), (2) and the above observation, it is clear that $Z(\mu)$ satisfies the 
recurrence relation (\ref{eq:Dev}).
\end{proof}

\begin{example}
\label{ex:E0321}
We consider the coefficient of $E_1E_2E_3E_0$ in $P_3$.
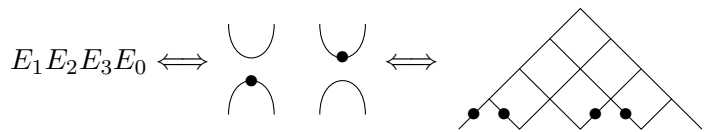
\begin{align*}
E_1E_2E_3E_0 \Longleftrightarrow
\tikzpic{-0.5}{[scale=0.6]
\draw(0,0)..controls(0,1)and(1,1)..(1,0)(0,2)..controls(0,1)and(1,1)..(1,2);
\draw(0.5,0.72)node{$\bullet$};
\draw(2,0)..controls(2,1)and(3,1)..(3,0)(2,2)..controls(2,1)and(3,1)..(3,2);
\draw(2.5,1.25)node{$\bullet$};
}\Longleftrightarrow
\tikzpic{-0.5}{[scale=0.4]
\draw(0,0)--(1,1)--(2,0)--(3,1)--(4,0)--(5,1)--(6,0)--(7,1)--(8,0);
\draw(1,1)--(4,4)--(7,1);
\draw(2,2)--(3,1)--(4,2)--(5,1)(3,3)--(4,2)--(5,3)(5,1)--(6,2);
\draw(0.5,0.5)node{$\bullet$}(1.5,0.5)node{$\bullet$};
\draw(4.5,0.5)node{$\bullet$}(5.5,0.5)node{$\bullet$};
}
\end{align*}
We have two admissible Dyck tilings.
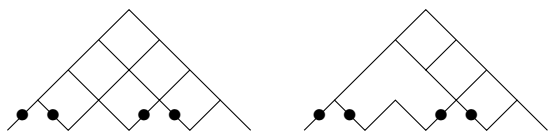
\begin{align*}
\tikzpic{-0.5}{[scale=0.4]
\draw(0,0)--(1,1)--(2,0)--(3,1)--(4,0)--(5,1)--(6,0)--(7,1)--(8,0);
\draw(1,1)--(4,4)--(7,1);
\draw(2,2)--(3,1)--(4,2)--(5,1)(3,3)--(4,2)--(5,3)(5,1)--(6,2);
\draw(0.5,0.5)node{$\bullet$}(1.5,0.5)node{$\bullet$};
\draw(4.5,0.5)node{$\bullet$}(5.5,0.5)node{$\bullet$};
}
\quad
\tikzpic{-0.5}{[scale=0.4]
\draw(0,0)--(1,1)--(2,0)--(3,1)--(4,0)--(5,1)--(6,0)--(7,1)--(8,0);
\draw(1,1)--(4,4)--(7,1);
\draw(3,3)--(4,2)--(5,3)(4,2)--(5,1)--(6,2);
\draw(0.5,0.5)node{$\bullet$}(1.5,0.5)node{$\bullet$};
\draw(4.5,0.5)node{$\bullet$}(5.5,0.5)node{$\bullet$};
}
\end{align*}
Then, the coefficient of this diagram is given by
\begin{align*}
\genfrac{}{}{}{}{[3]}{[6]}\genfrac{}{}{}{}{[2]}{[4]}\genfrac{}{}{}{}{1}{[2]}
+\genfrac{}{}{}{}{[3]}{[6]}\genfrac{}{}{}{}{1}{[2]}
=\genfrac{}{}{}{}{[3]^2}{[6][4]}.
\end{align*}

There are two non-admissible diagrams:
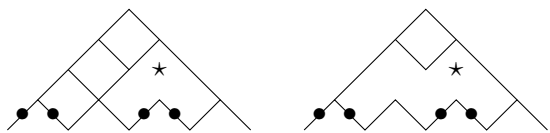
\begin{align*}
\tikzpic{-0.5}{[scale=0.4]
\draw(0,0)--(1,1)--(2,0)--(3,1)--(4,0)--(5,1)--(6,0)--(7,1)--(8,0);
\draw(1,1)--(4,4)--(7,1);
\draw(2,2)--(3,1)--(5,3)(3,3)--(4,2);
\draw(0.5,0.5)node{$\bullet$}(1.5,0.5)node{$\bullet$};
\draw(4.5,0.5)node{$\bullet$}(5.5,0.5)node{$\bullet$};
\draw(5,2)node{$\star$};
}
\quad
\tikzpic{-0.5}{[scale=0.4]
\draw(0,0)--(1,1)--(2,0)--(3,1)--(4,0)--(5,1)--(6,0)--(7,1)--(8,0);
\draw(1,1)--(4,4)--(7,1);
\draw(3,3)--(4,2)--(5,3);
\draw(0.5,0.5)node{$\bullet$}(1.5,0.5)node{$\bullet$};
\draw(4.5,0.5)node{$\bullet$}(5.5,0.5)node{$\bullet$};
\draw(5,2)node{$\star$};
}
\end{align*}
The Dyck tile with $\star$ violates the condition. The dotted cap below $\star$ 
has a Dyck tile of size $1$.
\end{example}

\subsection{The diagram with a unique dot}
\label{sec:DTodd}
We introduce two colors on a vertical Hermite history.
Recall that a line in an Hermite history is called a trajectory. In other words, an Hermite history
consists of several trajectories.
As in Remark \ref{remark:traj}, we have $n$ trajectories in an Hermite history if we 
allow a trajectory of length zero.

We consider the following conditions on a vertical Hermite history.
\begin{enumerate}[(Q1)]
\item There exists a unique trajectory with the color green. The length of this trajectory is non-zero.	
\item The trajectories left to the green trajectory are all red.
If the left-most trajectory is green, then there is no red trajectories with length non-zero.
\item There are no trajectories on Dyck tiles right to the green trajectory. That is, we have a Dyck 
tiling $T$ right to the green trajectory but remove all the trajectories on the Dyck tiles in $T$.
\end{enumerate}

The condition (Q3) implies that there may be a region without trajectories.
Such a region may be empty.

We impose one more condition on the vertical Hermite history.
By Remark \ref{remark:traj}, we have $n$ vertical trajectories in a vertical Hermite 
history if we include trajectories of length zero.
Fix a Dyck tiling and its vertical Hermite history. 
We denote by $T_{\ge0}$ the set of vertical
trajectories of size weakly larger than zero.
Suppose that the unique green trajectory $t$ is the $m$-th vertical trajectory
from left in $T_{\ge0}$. 
Let $t_i$, $1\le i\le n-m$, be the $m+i$-th vertical trajectory from left in $T_{\ge0}$.
We denote by $n_{i}$ the number of Dyck tiles in the trajectory $t_{i}$.
These $n-m$ vertical trajectories are removed by (Q3).
We consider the following condition:
\begin{enumerate}
\item[(Q4)]
$n_{1}\ge n_{2}\ge\ldots\ge n_{n-m}$.
\end{enumerate}
Note that $n_{n-m}=0$ in (Q4) since the lower Dyck path and the top Dyck path 
in a Dyck tiling share the same right-most down  step and we have a trajectory
of length zero. 

\begin{defn}
We call a vertical Hermite history which satisfies the conditions (Q1) to (Q4)
a {\it bi-colored vertical Hermite history}.
\end{defn}

We first study the condition (Q4).
Fix a Dyck tiling $T$ of size $n$ above a Dyck path $\lambda$ and below the top Dyck path, 
and its bi-colored vertical Hermite history $V$.
Suppose that the $m$-th trajectory from left in $V$ is the unique green trajectory.
We introduce the operation which we call a deletion of $T$.
Let $d$ be the left-most down step of the lowest Dyck tile in the first trajectory.
Let $\lambda_{i+1}$ be the down step of $\lambda$ which is obtained by moving 
$d$ downward. Since $\lambda_{i+1}$ is a down step and $d$ is in the first trajectory,
the step $\lambda_{i}$ is an up step in $\lambda$.
We delete all the Dyck tiles in the first vertical trajectory and shrink all the Dyck tiles 
above $\lambda_{i}$ and $\lambda_{i+1}$.
The newly obtained Dyck tiling has a lower Dyck path $\lambda'$ and 
the upper Dyck path $U^{n-1}D^{n-1}$.
Here, the Dyck path $\lambda'$ is obtained from $\lambda$ by deleting $\lambda_{i}$ and $\lambda_{i+1}$.
This operation is visualized in Figure \ref{fig:delDT}.
\begin{figure}[ht]
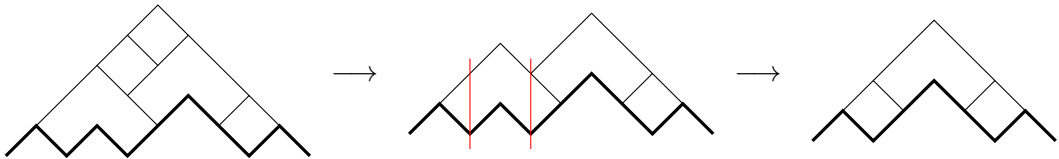

\tikzpic{-0.5}{[scale=0.4]
\draw[very thick](0,0)--(1,1)--(2,0)--(3,1)--(4,0)--(6,2)--(8,0)--(9,1)--(10,0);
\draw(1,1)--(3,3)--(5,1)(4,2)--(6,4)--(9,1)(8,2)--(7,1);
\draw(3,3)--(5,5)--(6,4)(4,4)--(5,3);
}
$\longrightarrow$
\tikzpic{-0.5}{[scale=0.4]
\draw[very thick](0,0)--(1,1)--(2,0)--(3,1)--(4,0)--(6,2)--(8,0)--(9,1)--(10,0);
\draw(1,1)--(3,3)--(5,1)(4,2)--(6,4)--(9,1)(8,2)--(7,1);
\draw[red](2,-0.5)--(2,2.5)(4,-0.5)--(4,2.5);
}
$\longrightarrow$
\tikzpic{-0.5}{[scale=0.4]
\draw[very thick](0,0)--(1,1)--(2,0)--(4,2)--(6,0)--(7,1)--(8,0);
\draw(1,1)--(2,2)--(3,1)(2,2)--(4,4)--(7,1)(6,2)--(5,1);
}
\caption{A deletion of a Dyck tiling}
\label{fig:delDT}
\end{figure}
We call this operation a {\it deletion} of $T$.
We perform the deletion of $T$ $m$ times.
Then, we obtain a Dyck tiling $T'$ of size $n-m$.
The condition (Q4) is equivalent to the following condition on $T'$: 
\begin{enumerate}
\item[(Q$4'$)]
$T'$ consists of only trivial Dyck tiles.
\end{enumerate}
This condition captures the property of (Q4) in terms of a Dyck tiling.
We make use of this equivalence between (Q4) and (Q$4'$) in 
the proof of Theorem \ref{thrm:DoddZ} below.

\begin{remark}
The equivalence between (Q4) and (Q$4'$) holds if and only if the upper path 
of a Dyck tiling is the top Dyck path $U^nD^n$.
Let us consider a Dyck tiling whose upper path is the top Dyck path.
Then, by Proposition \ref{prop:traj}, any vertical trajectory 
ends at a down step in the upper path of a Dyck tiling. 
Since the upper path of a Dyck tiling is the top Dyck path, 
it is obvious that $n_{i}<n_{i+1}$ gives a Dyck tiling 
with non-trivial Dyck tiles. A non-trivial Dyck tile appears only in this way.
\end{remark}

\begin{example}
Consider the Dyck tiling in Figure \ref{fig:Hh}.
We have the following unique vertical Hermite history with two colors.
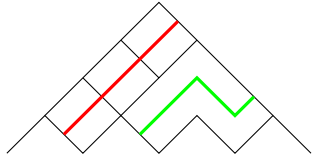
\begin{align*}
\tikzpic{-0.5}{[scale=0.5]
\draw(0,0)--(1,1)--(2,0)--(3,1)--(4,0)--(5,1)--(6,0)--(7,1)--(8,0);
\draw(1,1)--(4,4)--(7,1);
\draw(2,2)--(3,1)--(5,3)(3,3)--(4,2);
\draw[very thick,red](1.5,0.5)--(4.5,3.5);
\draw[very thick,green](3.5,0.5)--(5,2)--(6,1)--(6.5,1.5);
}
\end{align*}
The left-most vertical trajectory (which is a red trajectory in the figure above) 
in the Dyck tiling cannot be green.
If the first vertical trajectory is green, then we have 
$(n_1,n_2,n_3)=(0,1,0)$, which violates the condition (Q4). 
\end{example}

Let $T$ be a Dyck tiling with bi-colored trajectories
satisfying the conditions (Q1) to (Q4), and $T^{A}$ be the Dyck tiling of the region
without trajectories in $T$.
We are allowed to consider the Dyck tilings such that it may have larger Dyck tiles than the Dyck tiles
in $T^{A}$. The Dyck tiles are within the region without trajectories in $T$.
Then, we define the weight $\mathrm{wt}(T^A)$ of $T^{A}$ by
\begin{align}
\label{eq:wtDA}
\mathrm{wt}(T^A):=\sum_{T'\in \mathcal{T}^A}\prod_{d\in T'}\genfrac{}{}{}{}{[h(d)]}{[h(d)+1]},
\end{align}
where $\mathcal{T}^A$ is the set of Dyck tilings in the region without trajectories in $T$, and 
$d$ is a Dyck tile in $T'\in\mathcal{T}^A$.

Let $t_{r}$ be a red vertical trajectory, and we define the top and bottom heights $h_{t}(t_r)$ 
and $h_{b}(t_r)$ of $t_r$  by the height of the top Dyck tile and the bottom Dyck tile in $t_r$
respectively.
Similarly, $t_{g}$ be a green vertical trajectory. The heights $h_{t}(t_g)$ and $h_{b}(t_g)$ 
are defined in a similar way.

Let $\lambda\le\mu$ be two Dyck paths. 
Let $\mathcal{D}(\lambda,\mu)$ be the set of bi-colored vertical Hermite histories of Dyck tilings 
above $\lambda$ and below $\mu$.
Suppose that we have a Dyck tiling $T^{A}$ in the region without trajectories in the vertical 
Hermite history.
\begin{defn}
We define 
\begin{align}
\label{eq:DTD}
Z'(\lambda,\mu):=\sum_{D\in\mathcal{D}(\lambda,\mu)} 2^{N(t_r)}\mathrm{wt}(T^{A})
\prod_{t_r}\genfrac{}{}{}{}{[h_{t}(t_r)][2 (h_{b}(t_r)-1)]}{[2 h_{t}(t_r)][h_{b}(t_r)-1]}
\prod_{t_g}\genfrac{}{}{}{}{[h_{t}(t_g)][h_{t}(t_g)+1-h_{b}(t_g)]}{[2 h_{t}(t_g)][h_{t}(t_g)+1]},
\end{align}
where $N(t_r)$ be the number of red trajectories $t_r$ such that $h_{b}(t_r)=1$.
When $h_{b}(t_r)=1$, we define $\genfrac{}{}{}{}{[2(h_b(t_r)-1)]}{[h_{b}(t_r)-1]}:=1$.
\end{defn}

Let $D^{\mathrm{odd}}$ be an $n+1$-strand diagram with a single dot, and $\lambda$ is the Dyck path
corresponding to the diagram $\mathtt{For}(D^{\mathrm{odd}})$.
Then, we have the following theorem which is one of the main results.
\begin{theorem} 
\label{thrm:DoddZ}
Let $Z'(\lambda,\mu)$ be as above, and $\mu_{0}$ the top Dyck path of size $n+1$.
The coefficient $\mathrm{coef}(D^{\mathrm{odd}})$ of the diagram $D^{\mathrm{odd}}$ is given
by
\begin{align}
\label{eq:DoddZ}
\mathrm{coef}(D^{\mathrm{odd}})=Z'(\lambda,\mu_{0}).
\end{align} 
\end{theorem}
\begin{proof}
We will show that the generating function $Z'(\lambda,\mu_{0})$ satisfies 
the recurrence relation (\ref{eq:DinDAD}) in Proposition \ref{prop:DinDAD}
if a bi-colored vertical Hermite history for $Z'(\lambda,\mu_0)$ satisfies the conditions from (Q1) to (Q4).
As in the proof of Theorem \ref{thrm:Deven}, we consider the trajectory 
$T_{d}$ which is left-most in the bi-colored vertical Hermite history. 
We first show that a bi-colored vertical Hermite history satisfies the conditions (Q1) to (Q3).
We consider two cases: 1) $T_{d}$ is a green trajectory, and
2) $T_{d}$ is a red trajectory.
\begin{enumerate}
\item Since the color of $T_{d}$ is green, we have no green or red trajectories right to $T_{d}$.
We may have a Dyck tiling without colored trajectories right to $T_{d}$.
From Eq. (\ref{eq:DinDAD}), this corresponds to the case such that $T_{d}$ gives the factor
$\mathrm{coef}(h_{n,i})$, and the region without colored trajectories corresponds to the 
generating function of Dyck tilings of type $A$.
This is achieved by choosing $i\in I(D)$, and the first sum in the right hand side of Eq. (\ref{eq:DinDAD})
is given. 
\item Since the color of $T_{d}$ is red, we have red trajectories and a unique right-most green trajectory 
right to $T_{d}$.
The recurrence relation (\ref{eq:DinDAD}) implies that if we have a coefficient $\mathrm{coef}(g_{n,i})$, 
then we have a diagram $D_{i}^{\mathrm{odd}}$. Since this diagram $D_{i}^{\mathrm{odd}}$ 
contains a unique dot, we apply the recurrence relation (\ref{eq:DinDAD}) once again. 
This procedure implies that unless we have 
a coefficient $\mathrm{coef}(g_{n,i})$, we recursively apply the recurrence relation (\ref{eq:DinDAD}).
This stops when we have $\mathrm{coef}(h_{n,i})$ and produces the diagram $D^{A}_{i}$.
This means that there exists a unique green trajectory $T_{g}$, and the trajectories left to 
$T_{g}$ are all red. Then, the region right to the trajectory $T_{g}$ does not have colored trajectories.
We have a factor $2$ when $i=1\in I(D)$, which comes from the fact that we have 
$\mathrm{coef}(h_{n,0})=\mathrm{coef}(g_{n,1})$ in Eq. (\ref{eq:Q2}) in Proposition \ref{prop:Q1}.
By combining these together, we have the second term in Eq. (\ref{eq:DinDAD}).
\end{enumerate}
From these, a bi-colored vertical Hermite history in the generating function 
satisfies the conditions (Q1) to (Q3).

Below, we show that a bi-colored vertical Hermite history satisfies the condition (Q4).
Consider the first sum in the recurrence relation (\ref{eq:DinDAD}).
Since $D^{A}_{i}$ is obtained from $D$ by removing an inner-most cap, 
the size of $D^{A}_{i}$ is one smaller than that of $D$.
Recall that the diagram $D^{A}_{i}$ is of type $A$.
If the lower Dyck path corresponding to $D^{A}_{i}$ is $\lambda$, we have to 
consider all the Dyck tilings above $\lambda$ and below the top Dyck path.

The expression (\ref{eq:wtDA}) implies that we need to have the Dyck tiling $T$
which consists of only trivial Dyck tiles.  
If the top and the bottom Dyck paths are given, we have a unique Dyck tiling with only trivial Dyck tiles,
and other Dyck tilings are with non-trivial Dyck tiles.
Such other Dyck tilings are obtained from $T$ by merging Dyck tiles into a larger Dyck tile.

Suppose that a Dyck tiling $T'$ has $m\ge0$ red trajectories and a unique green trajectory.
Then, since the size of the Dyck tiling $T'$ is $n+1$, the size of the region without colored trajectories
is to be $n+1-(m+1)=n-m$.
The observation in the previous paragraph implies that if we perform the deletion of the Dyck tiling $T'$ $m+1$ times, 
then we need to have a Dyck tiling which consists of only trivial Dyck tiles.
This is nothing but (Q$4'$).
By the equivalence between (Q4) and (Q$4'$), a bi-colored vertical Hermite 
history satisfies the condition (Q4).
On the contrary, suppose that a bi-colored vertical Hermite history violates the condition (Q4). 
Then, the region without colored trajectories consists of a Dyck tiling with 
a non-trivial Dyck tile after the $m+1$ deletions.
In this case, we can not have the Dyck tiling consisting $T$ of only trivial Dyck tiles
when we calculate the weight (\ref{eq:wtDA}) of $D^{A}_{i}$.
From these observations, it is clear that bi-colored vertical
Hermite histories in the generating function contain only the trajectories which satisfy the condition (Q4).

As a consequence, a bi-colored vertical Hermite history in $Z'(\lambda,\mu_{0})$ 
satisfies the conditions from (Q1) to (Q4), and 
$Z'(\lambda,\mu_0)$ also satisfies the recurrence relation (\ref{eq:DinDAD}).
These imply Eq. (\ref{eq:DoddZ}), which completes the proof.
\end{proof}

\begin{example}
We calculate the coefficient of $E_0E_1$ in $Q_3$.
The diagram for $E_0E_1$ and its corresponding top and bottom Dyck paths are as follows.
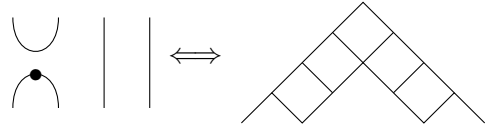
\begin{align*}
\tikzpic{-0.5}{[scale=0.6]
\draw(0,0)..controls(0,1)and (1,1)..(1,0)(0,2)..controls(0,1)and(1,1)..(1,2);
\draw(0.5,0.72)node{$\bullet$};
\draw(2,0)--(2,2)(3,0)--(3,2);
}\Longleftrightarrow
\tikzpic{-0.5}{[scale=0.4]
\draw(0,0)--(1,1)--(2,0)--(4,2)--(6,0)--(7,1)--(8,0);
\draw(1,1)--(4,4)--(7,1);
\draw(2,2)--(3,1)(3,3)--(4,2)--(5,3)(5,1)--(6,2);
}
\end{align*}

We have six bi-colored vertical Hermite histories.
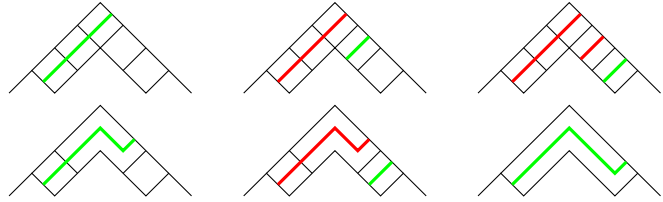
\begin{align*}
\tikzpic{-0.5}{[scale=0.3]
\draw(0,0)--(1,1)--(2,0)--(4,2)--(6,0)--(7,1)--(8,0);
\draw(1,1)--(4,4)--(7,1);
\draw(2,2)--(3,1)(3,3)--(4,2)--(5,3)(5,1)--(6,2);
\draw[very thick,green](1.5,0.5)--(4.5,3.5);
}\quad
\tikzpic{-0.5}{[scale=0.3]
\draw(0,0)--(1,1)--(2,0)--(4,2)--(6,0)--(7,1)--(8,0);
\draw(1,1)--(4,4)--(7,1);
\draw(2,2)--(3,1)(3,3)--(4,2)--(5,3)(5,1)--(6,2);
\draw[very thick,red](1.5,0.5)--(4.5,3.5);
\draw[very thick,green](4.5,1.5)--(5.5,2.5);
}\quad
\tikzpic{-0.5}{[scale=0.3]
\draw(0,0)--(1,1)--(2,0)--(4,2)--(6,0)--(7,1)--(8,0);
\draw(1,1)--(4,4)--(7,1);
\draw(2,2)--(3,1)(3,3)--(4,2)--(5,3)(5,1)--(6,2);
\draw[very thick,red](1.5,0.5)--(4.5,3.5);
\draw[very thick,red](4.5,1.5)--(5.5,2.5);
\draw[very thick,green](5.5,0.5)--(6.5,1.5);
}\\
\tikzpic{-0.5}{[scale=0.3]
\draw(0,0)--(1,1)--(2,0)--(4,2)--(6,0)--(7,1)--(8,0);
\draw(1,1)--(4,4)--(7,1);
\draw(2,2)--(3,1)(5,1)--(6,2);
\draw[very thick,green](1.5,0.5)--(4,3)--(5,2)--(5.5,2.5);
}\quad
\tikzpic{-0.5}{[scale=0.3]
\draw(0,0)--(1,1)--(2,0)--(4,2)--(6,0)--(7,1)--(8,0);
\draw(1,1)--(4,4)--(7,1);
\draw(2,2)--(3,1)(5,1)--(6,2);
\draw[very thick,red](1.5,0.5)--(4,3)--(5,2)--(5.5,2.5);
\draw[very thick,green](5.5,0.5)--(6.5,1.5);
}\quad
\tikzpic{-0.5}{[scale=0.3]
\draw(0,0)--(1,1)--(2,0)--(4,2)--(6,0)--(7,1)--(8,0);
\draw(1,1)--(4,4)--(7,1);
\draw[very thick,green](1.5,0.5)--(4,3)--(6,1)--(6.5,1.5);
}
\end{align*}
The sum of the weights of the above diagrams is given by
\begin{align*}
\genfrac{}{}{}{}{[3]^{2}}{[6][4]}\genfrac{}{}{}{}{[2]}{[3]}\genfrac{}{}{}{}{1}{[2]}
+2\genfrac{}{}{}{}{[3]}{[6]}\genfrac{}{}{}{}{[2]}{[4][3]}\genfrac{}{}{}{}{1}{[2]}
+2\genfrac{}{}{}{}{[3]}{[6]}\genfrac{}{}{}{}{[2][2]}{[4][1]}\genfrac{}{}{}{}{1}{[2]^2}
+\genfrac{}{}{}{}{[2]^2}{[4][3]}\genfrac{}{}{}{}{1}{[2]}
+2 \genfrac{}{}{}{}{[2]}{[4]}\genfrac{}{}{}{}{1}{[2]^2}+\genfrac{}{}{}{}{1}{[2]^2}
=\genfrac{}{}{}{}{[3]^3}{[6][4]}.
\end{align*}
\end{example}

\begin{example}
\label{ex:013}
We calculate the coefficient of $E_0E_1E_3$ in $Q_3$. 
The diagram for $E_0E_1E_3$ and its corresponding top and bottom Dyck paths are 
given by 
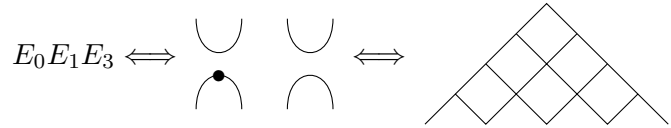
\begin{align*}
E_0E_1E_3 \Longleftrightarrow
\tikzpic{-0.5}{[scale=0.6]
\draw(0,0)..controls(0,1)and(1,1)..(1,0)(0,2)..controls(0,1)and(1,1)..(1,2);
\draw(0.5,0.72)node{$\bullet$};
\draw(2,0)..controls(2,1)and(3,1)..(3,0)(2,2)..controls(2,1)and(3,1)..(3,2);
}\Longleftrightarrow
\tikzpic{-0.5}{[scale=0.4]
\draw(0,0)--(1,1)--(2,0)--(3,1)--(4,0)--(5,1)--(6,0)--(7,1)--(8,0);
\draw(1,1)--(4,4)--(7,1);
\draw(2,2)--(3,1)--(4,2)--(5,1)(3,3)--(4,2)--(5,3)(5,1)--(6,2);
}
\end{align*}
We have eight bi-colored  vertical Hermite histories.
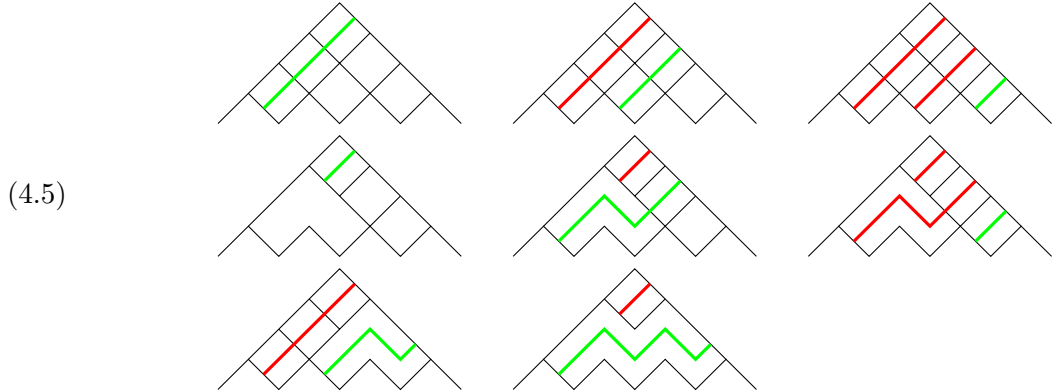
\begin{align}
\label{eq:8vHh}
\begin{aligned}
&\tikzpic{-0.5}{[scale=0.4]
\draw(0,0)--(1,1)--(2,0)--(3,1)--(4,0)--(5,1)--(6,0)--(7,1)--(8,0);
\draw(1,1)--(4,4)--(7,1);
\draw(2,2)--(3,1)--(4,2)--(5,1)(3,3)--(4,2)--(5,3)(5,1)--(6,2);
\draw[very thick,green](1.5,0.5)--(4.5,3.5);
}\quad
\tikzpic{-0.5}{[scale=0.4]
\draw(0,0)--(1,1)--(2,0)--(3,1)--(4,0)--(5,1)--(6,0)--(7,1)--(8,0);
\draw(1,1)--(4,4)--(7,1);
\draw(2,2)--(3,1)--(4,2)--(5,1)(3,3)--(4,2)--(5,3)(5,1)--(6,2);
\draw[very thick,red](1.5,0.5)--(4.5,3.5);
\draw[very thick,green](3.5,0.5)--(5.5,2.5);
}\quad
\tikzpic{-0.5}{[scale=0.4]
\draw(0,0)--(1,1)--(2,0)--(3,1)--(4,0)--(5,1)--(6,0)--(7,1)--(8,0);
\draw(1,1)--(4,4)--(7,1);
\draw(2,2)--(3,1)--(4,2)--(5,1)(3,3)--(4,2)--(5,3)(5,1)--(6,2);
\draw[very thick,red](1.5,0.5)--(4.5,3.5);
\draw[very thick,red](3.5,0.5)--(5.5,2.5);
\draw[very thick,green](5.5,0.5)--(6.5,1.5);
} \\
&\tikzpic{-0.5}{[scale=0.4]
\draw(0,0)--(1,1)--(2,0)--(3,1)--(4,0)--(5,1)--(6,0)--(7,1)--(8,0);
\draw(1,1)--(4,4)--(7,1);
\draw(3,3)--(5,1)--(6,2)(5,3)--(4,2);
\draw[very thick,green](3.5,2.5)--(4.5,3.5);
}\quad
\tikzpic{-0.5}{[scale=0.4]
\draw(0,0)--(1,1)--(2,0)--(3,1)--(4,0)--(5,1)--(6,0)--(7,1)--(8,0);
\draw(1,1)--(4,4)--(7,1);
\draw(3,3)--(5,1)--(6,2)(5,3)--(4,2);
\draw[very thick,red](3.5,2.5)--(4.5,3.5);
\draw[very thick,green](1.5,0.5)--(3,2)--(4,1)--(5.5,2.5);
}\quad
\tikzpic{-0.5}{[scale=0.4]
\draw(0,0)--(1,1)--(2,0)--(3,1)--(4,0)--(5,1)--(6,0)--(7,1)--(8,0);
\draw(1,1)--(4,4)--(7,1);
\draw(3,3)--(5,1)--(6,2)(5,3)--(4,2);
\draw[very thick,red](3.5,2.5)--(4.5,3.5);
\draw[very thick,red](1.5,0.5)--(3,2)--(4,1)--(5.5,2.5);
\draw[very thick,green](5.5,0.5)--(6.5,1.5);
}\\
&\tikzpic{-0.5}{[scale=0.4]
\draw(0,0)--(1,1)--(2,0)--(3,1)--(4,0)--(5,1)--(6,0)--(7,1)--(8,0);
\draw(1,1)--(4,4)--(7,1);
\draw(2,2)--(3,1)--(5,3)(3,3)--(4,2);
\draw[very thick,red](1.5,0.5)--(4.5,3.5);
\draw[very thick,green](3.5,0.5)--(5,2)--(6,1)--(6.5,1.5);
}\quad
\tikzpic{-0.5}{[scale=0.4]
\draw(0,0)--(1,1)--(2,0)--(3,1)--(4,0)--(5,1)--(6,0)--(7,1)--(8,0);
\draw(1,1)--(4,4)--(7,1);
\draw(3,3)--(4,2)--(5,3);
\draw[very thick,red](3.5,2.5)--(4.5,3.5);
\draw[very thick,green](1.5,0.5)--(3,2)--(4,1)--(5,2)--(6,1)--(6.5,1.5);
}
\end{aligned}
\end{align}
The sum of the weight is given by
\begin{align*}
&\genfrac{}{}{}{}{[3]^2}{[6][4]}\left(\genfrac{}{}{}{}{1}{[3][2]}+\genfrac{}{}{}{}{1}{[2]}\right)
+2\genfrac{}{}{}{}{[3]}{[6]}\genfrac{}{}{}{}{[2]^2}{[4][3]}\genfrac{}{}{}{}{1}{[2]}
+2^{2}\genfrac{}{}{}{}{[3]}{[6]}\genfrac{}{}{}{}{[2]}{[4]}\genfrac{}{}{}{}{1}{[2]^2} \\
&+\genfrac{}{}{}{}{[3]}{[6][4]}\left(\genfrac{}{}{}{}{1}{[3][2]}+\genfrac{}{}{}{}{1}{[2]}\right)
+\genfrac{}{}{}{}{[3][4]}{[6][2]}\genfrac{}{}{}{}{[2]^2}{[4][3]}\genfrac{}{}{}{}{1}{[2]}
+2\genfrac{}{}{}{}{[3][4]}{[6][2]}\genfrac{}{}{}{}{[2]}{[4]}\genfrac{}{}{}{}{1}{[2]^2} \\
&+2\genfrac{}{}{}{}{[3]}{[6]}\genfrac{}{}{}{}{1}{[2]^2}
+\genfrac{}{}{}{}{[3][4]}{[6][2]}\genfrac{}{}{}{}{1}{[2]^2} \\
&=\genfrac{}{}{}{}{[2]^3}{[6][4]}\left([2]^2+1\right).
\end{align*}

There are two non-admissible bi-colored Hermite histories of Dyck tilings:
\begin{align}
\label{eq:naDyck}
\tikzpic{-0.5}{[scale=0.4]
\draw(0,0)--(1,1)--(2,0)--(3,1)--(4,0)--(5,1)--(6,0)--(7,1)--(8,0);
\draw(1,1)--(4,4)--(7,1);
\draw(2,2)--(3,1)--(5,3)(3,3)--(4,2);
\draw[very thick,green](1.5,0.5)--(4.5,3.5);
}\quad
\tikzpic{-0.5}{[scale=0.4]
\draw(0,0)--(1,1)--(2,0)--(3,1)--(4,0)--(5,1)--(6,0)--(7,1)--(8,0);
\draw(1,1)--(4,4)--(7,1);
\draw(3,3)--(4,2)--(5,3);
\draw[very thick,green](3.5,2.5)--(4.5,3.5);
}
\end{align}
The both bi-colored Hermite histories have three (removed) trajectories 
right to the green trajectory.
Then, the numbers of Dyck tiles in these three trajectories are  
$(n_1,n_2,n_3)=(0,1,0)$ from left to right if we apply the deletion of the Dyck tilings.
They are not admissible since $n_1<n_2$ violates (Q4).	
\end{example}

In what follows, we consider the practical reason why the condition (Q4) on a 
bi-colored vertical Hermite history 
is necessary with an example.

For simplicity, we consider the two non-admissible Dyck tilings in Eq. (\ref{eq:naDyck})
of Example \ref{ex:013}.
The left non-admissible Dyck tiling: 
The green vertical trajectory corresponds to the coefficient $\mathrm{coef}(h_{3,1})$.
Then, diagrammatically, we calculate  
\begin{align}
\label{eq:redDyck}
\tikzpic{-0.5}{[scale=0.5]
\draw[green](0,0)..controls(0,1)and(1,1)..(1,0);
\draw(0,2)..controls(0,1)and(1,1)..(1,2);
\draw(2,0)..controls(2,1)and(3,1)..(3,0)(2,2)..controls(2,1)and(3,1)..(3,2);
}_{D}
\longrightarrow \mathrm{coef}(h_{3,1}) \cdot 
\tikzpic{-0.5}{[scale=0.5]
\draw(0,0)..controls(0,1)and(1,1)..(1,0);
\draw(0,2)..controls(0,1)and(1,1)..(1,2);
\draw(2,0)--(2,2);
}_{A},
\end{align}
where the green cap corresponds to the green trajectory.
The right-hand side of Eq. (\ref{eq:redDyck}) implies that 
we have a diagram of type $A$ with three strands.
We compare this situation with the left Dyck tiling in Eq. (\ref{eq:naDyck}).
This Dyck tiling contains three trivial Dyck tiles on the green trajectory, 
and a unique non-trivial Dyck tile of size $1$.
This non-trivial Dyck tile has no trajectory, and this means that 
this tile corresponds to type $A$. 
The region of type $A$ corresponds to a single tile at height $1$ since this tile
is a non-trivial Dyck tile. 
By shrinking the Dyck tile, it is easy to see that this Dyck tile corresponds to 
the following diagram of type $A$:
\begin{align*}
\tikzpic{-0.5}{[scale=0.5]
\draw(0,0)..controls(0,1)and(1,1)..(1,0);
\draw(0,2)..controls(0,1)and(1,1)..(1,2);
}_{A}
\end{align*} 
This contains only two strands and does not fit to the situation in Eq. (\ref{eq:redDyck}).

The right Dyck tiling in Eq. (\ref{eq:naDyck}) is not admissible in the same reason.

On the contrary, we consider the left-most admissible Dyck tiling in the second row
in Eq. (\ref{eq:8vHh}).
The green trajectory corresponds to $\mathrm{coef}(h_{3,3})$. 
As a diagram, we calculate 
\begin{align}
\label{eq:redDyck2}
\tikzpic{-0.5}{[scale=0.5]
\draw(0,0)..controls(0,1)and(1,1)..(1,0);
\draw(0,2)..controls(0,1)and(1,1)..(1,2);
\draw[green](2,0)..controls(2,1)and(3,1)..(3,0);
\draw(2,2)..controls(2,1)and(3,1)..(3,2);
}_{D}
\longrightarrow \mathrm{coef}(h_{3,3}) \cdot 
\tikzpic{-0.5}{[scale=0.5]
\draw(0,0)..controls(0,1)and(1,1)..(1,0);
\draw(0,2)..controls(0,1)and(1,1)..(1,2);
\draw(2,0)--(2,2);
}_{A}.
\end{align}
The Dyck tiling contains one Dyck tile with a green trajectory, and 
three Dyck tiles without trajectories.
These three Dyck tiles consist of two Dyck tiles of height $1$ and 
one Dyck tiles of height $2$.
This diagram of type $A$ indeed corresponds to the diagram in the right-hand side
of Eq. (\ref{eq:redDyck2}). 
Therefore, this Dyck tiling is admissible.

\section{Some explicit expressions of the coefficients}
\label{sec:ex}
In this section, we give some explicit expressions of the coefficients of 
the elements in Jones--Wenzl projections. 
We write $\mathrm{coef}^{n}(A)$ as a coefficient of $A$ in the projection $Q_{n}$.

We first study the symmetry of the coefficients of the diagram in 
the Jones--Wenzl projection $Q_n$.

Let $D$ be an $n+1$-strand Temperley--Lieb diagram in $\mathrm{TL}^{D}_{n+1}$.
The diagram $D$ has an algebraic representation, that is, $D$ corresponds 
to the word $E_{D}:=E_{s_1}\cdots E_{s_{m}}$ where $E_{D}$ is a reduced expression.
Given a word $E_{D}$, we introduce a transposed word 
$\overline{E_{D}}:=E_{s_{m}}\cdots E_{s_1}$.
\begin{prop}
\label{prop:barE}
Let $E_D$ and $\overline{E_{D}}$ be words as above.
Then, we have 
$\mathrm{coef}^{n}(E_{D})=\mathrm{coef}^{n}(\overline{E_{D}})$.
Further, the coefficient of $D$ in $Q_{n}$ is equal to that of $\overline{D}$, where
$\overline{D}$ is the upside down diagram obtained from $D$ (We change the position of 
a unique dot in $\overline{D}$ to satisfy the property (P2b) if $\overline{D}$ has a unique dot.). 
\end{prop} 
\begin{proof}
The recurrence relation (\ref{eq:Q2}) in Proposition \ref{prop:Q1} can be 
rewritten as 
\begin{align*}
Q_{n+1}=\left(\sum_{i=1}^{n+2}\mathrm{coef}(g_{n+1,i})\overline{g_{n+1,i}}
+\sum_{j=0}^{n+1}\mathrm{coef}(h_{n+1,j})\overline{h_{n+1,j}}
\right)Q_{n-1}.
\end{align*}
This implies that $\mathrm{coef}(E_{D})=\mathrm{coef}(\overline{E_D})$.
Since the diagram $\overline{D}$ is nothing but a diagram for $\overline{E_D}$,
the second statement follows.
\end{proof}

Let $E_{D}:=E_{s_1}\ldots E_{s_{m}}\in\mathrm{TL}^{D}_{n+1}$ be a word of $E_{i}$, $0\le i\le n$.
We define an involution $\theta:E_{D}\mapsto E_{D'}$ where $E_{D'}$ is obtained from 
$E_{D}$  by replacing $E_{0}$ by $E_1$ and $E_{1}$ by $E_{0}$ in $E_{D}$.
 
\begin{prop}
\label{prop:theta}
We have $\mathrm{coef}^{n}(E_{D})=\mathrm{coef}^{n}(E_{D'})$.
\end{prop}
\begin{proof}
Note that $\theta$ is an involution of the Dynkin diagram $D_{n+1}$ such that 
we exchange the labels $0$ and $1$ and fix other labels.
The statement immediately follows from this fact.
\end{proof}

We introduce the following summation formula which is used later.
\begin{lemma}
\label{lemma:df}
We have 
\begin{align*}
\sum_{k=j+1}^{n}\genfrac{}{}{}{}{[k][k-1]}{[2k][2(k-1)]}=\genfrac{}{}{}{}{[n][j][n-j]}{[2n][2j]}.
\end{align*}
\end{lemma}
\begin{proof}
By induction on $n$, one can easily prove the claim.
\end{proof}

We first consider some cases where the coefficient is expressed by a 
simple fraction of $q$-integers.
We first consider the reduced expressions without the product $E_0E_1$
in its algebraic representation.
\begin{prop}
\label{prop:Eji}
Let $i$ and $j$ be positive integers such that $1\le i\le j\le n$.
We have 
\begin{align}
\label{eq:Eji}
\mathrm{coef}^{n}(E_{j}E_{j-1}\ldots E_{i})
=
\begin{cases}
\displaystyle \genfrac{}{}{}{}{[n][n+1]}{[2n][2]}, & \text{ if } i=j=1, \\
\displaystyle \genfrac{}{}{}{}{[n][2(i-1)][n+1-j]}{[2n][i-1]}, & \text{ otherwise }. 
\end{cases}
\end{align}
\end{prop}
\begin{proof}
We prove the second cases in Eq. (\ref{eq:Eji}) by induction on $n$ since the first case can be proven in a similar way. 
By use of Eq. (\ref{eq:Dev}) in Proposition \ref{prop:Deven} and the induction hypothesis, 
we have 
\begin{align*}
\mathrm{coef}^{n}(E_{j}\cdots E_{i})
&=\genfrac{}{}{}{}{[n][2(i-1)]}{[2n][i-1]}\mathrm{coef}^{n-1}(E_{n-1}E_{n-2}\cdots E_{j})
+\mathrm{coef}^{n-1}(E_{j}\cdots E_{i}), \\
&=\genfrac{}{}{}{}{[n][n-1][2(i-1)][2(j-1)]}{[2n][2(n-1)][i-1][j-1]}
+\mathrm{coef}^{n-1}(E_{n-1}E_{n-2}\cdots E_{j}), \\
&=\left(\sum_{k=j+1}^{n}\genfrac{}{}{}{}{[k][k-1]}{[2k][2(k-1)]}\right)\genfrac{}{}{}{}{[2(i-1)][2(j-1)]}{[i-1][j-1]}
+\genfrac{}{}{}{}{[j][2(j-1)]}{[2j][j-1]}, \\
&=\genfrac{}{}{}{}{[n][2(i-1)][n+1-j]}{[2n][i-1]},
\end{align*}
where we have used Lemma \ref{lemma:df}.
\end{proof}

In what follows, we consider the coefficients of elements which contains
a product $E_0E_{1}$.

\begin{prop}
\label{prop:Ej01i}
Let $D=E_{j}E_{j-1}\cdots E_{1}E_{0}E_{2}\cdots E_{i}$ with $j\ge1$ and $i\ge 2$.
We have 
\begin{align}
\label{eq:Eji01}
\mathrm{coef}^{n}(D)=\genfrac{}{}{}{}{[n][n+1-i][n+1-j]}{[2n][n+1]}.
\end{align}
Further, we have 
\begin{align}
\label{eq:E01}
\mathrm{coef}^{n}(E_0E_1)=\genfrac{}{}{}{}{[n]^3}{[2n][n+1]}.
\end{align}
\end{prop}
\begin{proof}
We prove Eq. (\ref{eq:Eji01}) by induction on $n$. For $n=1,2$, one can show that Eq. (\ref{eq:Eji01})
holds by a simple calculation.

We make use of the recurrence relation (\ref{eq:recQ1}) in Proposition \ref{prop:recQ1}.
We calculate the contributions to $\mathrm{coef}^{n}(D)$ from the three terms in
the right hand side of Eq. (\ref{eq:recQ1}).
From the first term, we have $\mathrm{coef}^{n-1}(D)$.
Therefore, by induction hypothesis, we have 
\begin{align}
\label{eq:cont1}
\mathrm{coef}^{n-1}(D)=\genfrac{}{}{}{}{[n-1][n-j][n-i]}{[2n-2][n]}.
\end{align}

From the second term, we have three cases:
\begin{enumerate}
\item $D=(E_{j}\cdots E_1E_{0}E_2\cdots E_{n-1})\cdot E_{n}\cdot (E_{n-1}\cdots E_{i})$,
\item $D=(E_{j}\cdots E_1E_{0}E_2\cdots E_{n-1})\cdot E_{n}\cdot (E_{n-1}\cdots E_{1}E_{0}E_2\cdots E_{i})$
\item $D=(E_{j}E_{j+1}\cdots E_{n-1})\cdot E_{n}\cdot (E_{n-1}\cdots E_1E_0E_2\cdots E_{i})$,
\end{enumerate}
In the case of (2), we have an extra factor $-[2]$. 
By Eq. (\ref{eq:Eji}), the three cases (1) to (4) contribute to $Q_{n}$ as 
\begin{align}
\label{eq:cont2}
\begin{aligned}
\genfrac{}{}{}{}{[n][2n-2]}{[2n][n-1]}
\left\{\genfrac{}{}{}{}{[n-1][n-j]}{[2n-2][n]}\left(\genfrac{}{}{}{}{[n-1][2i-2]}{[2n-2][i-1]}
-[2]\genfrac{}{}{}{}{[n-1][n-i]}{[2n-2][n]}\right)\right. \\
\left.+\genfrac{}{}{}{}{[n-1][n-i]}{[2n-2][n]}\genfrac{}{}{}{}{[n-1][2j-2]}{[2n-2][j-1]}
\right\}.
\end{aligned}
\end{align}

Let $D_{j,i}:=E_{j}E_{j-1}\cdots E_{1} E_{0} E_{2}\cdots E_{i}$.
From the third term in Eq. (\ref{eq:recQ1}), 
we have $D=D'E_{w_n}D''$ where 
$D$ is either $E_{j}E_{j+1}\cdots E_{n-1}$ or $D_{j,n-1}$,
and $D''$ is either $E_{n-1}E_{n-2}\cdots E_{i}$ or $D_{n-1,i}$.
Note that if we choose $D_{j,n-1}$ or $D_{n-1,i}$, then 
$D'E_{w_n}D''$ has a factor $-[2]$.
Therefore, the contribution to $Q_{n}$ is given by
\begin{align}
\label{eq:cont3}
\genfrac{}{}{}{}{[n]}{[2n][n+1]}
\left(\genfrac{}{}{}{}{[n-1][2j-2]}{[2n-2][j-1]}-[2]\genfrac{}{}{}{}{[n-1][n-j]}{[2n-2][n]}\right)
\left(\genfrac{}{}{}{}{[n-1][2i-2]}{[2n-2][i-1]}-[2]\genfrac{}{}{}{}{[n-1][n-i]}{[2n-2][n]}\right).
\end{align}

By taking the sum of Eqs. (\ref{eq:cont1}), (\ref{eq:cont2}) and (\ref{eq:cont3}),
we obtain Eq. (\ref{eq:Eji01}).
Equation (\ref{eq:E01}) can be shown in a similar way.
\end{proof}

In some cases, the coefficient of an element in $Q_{n}$ is given by a simple fraction of 
$q$-integers as we have seen Propositions \ref{prop:Eji} and \ref{prop:Ej01i}.
However, in general, the coefficients are the sum of fractions of $q$-integers.
A typical example of such coefficients is $\mathrm{coef}^{n}(E_{0}E_1E_{3})$.

\begin{prop}
\label{prop:E013}
The coefficient $\mathrm{coef}^{n}(E_{0}E_1E_3)$ is given by
\begin{align}
\mathrm{coef}^{n}(E_{0}E_1E_3)
=\genfrac{}{}{}{}{[n-1]^3[n-2][2]^2}{[2n-2][n+1][n]}+2\genfrac{}{}{}{}{[n][n-1]^3[n-2][2]}{[2n][2n-2][n+1]}.
\end{align}
\end{prop}

Before proceeding to the proof of Proposition \ref{prop:E013}, we will 
calculate the necessary coefficients.
\begin{lemma}
\label{lemma:E120}
We have 
\begin{align*}
\mathrm{coef}^{n}(E_1E_2E_0)=\mathrm{coef}^{n}(E_0E_2E_1)=\genfrac{}{}{}{}{[n][n-1]}{[2n][2]}.
\end{align*}
\end{lemma}
\begin{proof}
We make use of Eq. (\ref{eq:Q2}) in Proposition \ref{prop:Q1}.
Since the word $E_{0}E_{2}E_{1}$ is rewritten as 
\begin{align*}
E_0E_2E_1=E_0E_2E_3\cdots E_{n-1}\cdot E_{n}E_{n-1}\cdots E_{1},
\end{align*}
we have 
\begin{align*}
\mathrm{coef}^{n}(E_0E_2E_1)&=\mathrm{coef}^{n-1}(E_0E_2\cdots E_{n-1})\mathrm{coef}^{n}(g_{n,1})
+\mathrm{coef}^{n-1}(E_0E_2E_1), \\
&=\genfrac{}{}{}{}{[n][n-1]}{[2n][2n-2]}+\mathrm{coef}^{n-1}(E_0E_1E_2), \\
&=\sum_{k=2}^{n}\genfrac{}{}{}{}{[k][k-1]}{[2k][2k-2]}, \\
&=\genfrac{}{}{}{}{[n][n-1]}{[2n][2]},
\end{align*}
where we have used Lemma \ref{lemma:df}. By applying $\theta$ on $E_{0}E_{2}E_{1}$, and 
Proposition \ref{prop:theta}, we have the equality $\mathrm{coef}^{n}(E_0E_2E_1)=\mathrm{coef}^{n}(E_1E_2E_0)$,
which completes the proof.
\end{proof}

\begin{lemma}
\label{lemma:En310}
We have 
\begin{align}
\label{eq:En310}
\mathrm{coef}^{n}(E_{n}E_{n-1}\cdots E_{3}E_{1}E_{0})
=\genfrac{}{}{}{}{[n-1]^2[2]}{[2n][n+1]}+\genfrac{}{}{}{}{[n-1]^3[2]^2}{[2n][2n-2]}.
\end{align}
\end{lemma}
\begin{proof}
We make use of Eq. (\ref{eq:Q2}).
The word $E_nE_{n-1}\cdots E_{3}E_1E_0$ can be expressed as products of two words
$H_{1}H_{2}$ where $H_1$ is in $Q_{n-1}$ and $H_{2}$ is either $g_{n,i}$ or $h_{n,j}$.
More precisely, we have 
if $H_2$ is either $h_{n,1}$ or $h_{n,3}$, then $H_{1}$ is either 
$E_0$, $E_1$, $E_1E_0$,$E_{1}E_2E_0$ or $E_{0}E_{2}E_{1}$.
Similarly, if $H_2$ is either $g_{n,1}$, $h_{n,0}$ or $g_{n,3}$, then 
$H_1$ is $E_0E_1$. 
By taking the sum of all contributions, we have Eq. (\ref{eq:En310}).
\end{proof}

\begin{lemma}
\label{lemma:gn31}
We have 
\begin{align*}
\mathrm{coef}^{n}(E_nE_{n-1}\cdots E_{3}E_1)=\mathrm{coef}^{n}(E_nE_{n-1}\cdots E_{3}E_0)
=\genfrac{}{}{}{}{[n]^2[n-1][3]}{[2n][2n-2][2]}.
\end{align*}
\end{lemma}
\begin{proof}
The word $E_{n}\cdots E_3E_1$ can be expressed as a product of two words $H_1H_{2}$
where $H_{1}=E_{1}$ and $H_{2}$ is either $g_{n,3}$ or $g_{n,1}$.
By Eq. (\ref{eq:Q2}), we have
\begin{align*}
\mathrm{coef}^{n}(E_n\cdots E_3E_1)=\mathrm{coef}^{n-1}(E_1)(\mathrm{coef}(g_{n,3})+\mathrm{coef}(g_{n,1}))
=\genfrac{}{}{}{}{[n]^2[n-1][3]}{[2n][2n-2][2]},
\end{align*}
where we have used Eq. (\ref{eq:Eji}).
\end{proof}

\begin{lemma}
\label{lemma:0gn1}
We have 
\begin{align}
\label{eq:E0n1}
\mathrm{coef}^{n}(E_0E_{n}E_{n-1}\cdots E_{1})=\mathrm{coef}^{n}(E_1E_{n}E_{n-1}\cdots E_{2}E_{0})
=\genfrac{}{}{}{}{[n]^2[n-1]}{[2n][2n-2][2]}.
\end{align}
\end{lemma}
\begin{proof}
The word $E_0E_{n}E_{n-1}\cdots E_{1}$ is uniquely written as a product of two words $H_1H_{2}$ where 
$H_{1}=E_0$ and $H_{2}=g_{n,1}$. By Eq. (\ref{eq:Q2}),
$\mathrm{coef}^{n}(E_{0}g_{n,1})=\mathrm{coef}^{n-1}(E_0)\mathrm{coef}^{n}(g_{n,1})$.
From Propositions \ref{prop:barE} and \ref{prop:Eji}, and the definition of $\mathrm{coef}^{n}(g_{n,1})$, 
we have Eq. (\ref{eq:E0n1}).
\end{proof}

\begin{lemma}
\label{lemma:021gn3}
We have 
\begin{align}
\label{eq:E021n3}
\mathrm{coef}^{n}(E_0E_2E_1g_{n,3})=\mathrm{coef}^{n}(E_1E_2E_0g_{n,3})
=\genfrac{}{}{}{}{[n][n-1][n-2][3]}{[2n][2n-2][2]}.
\end{align}
\end{lemma}
\begin{proof}
Note that the word $E_{0}E_2E_1g_{n,3}$ can be also expressed as $E_{0}E_2E_1g_{n,1}$.
Therefore, we have two expressions for $E_{0}E_{2}E_{1}g_{n,3}$.
By Eq. (\ref{eq:Q2}) and Lemma \ref{lemma:E120}, we have Eq. (\ref{eq:E021n3}).
\end{proof}

\begin{lemma}
\label{lemma:120gn1}
We have 
\begin{align*}
\mathrm{coef}^{n}(E_1E_2E_0 g_{n,1})=\mathrm{coef}^{n}(E_0E_2E_1h_{n,0})
=\genfrac{}{}{}{}{[n][n-1][n-2]}{[2n][2n-2][2]}.
\end{align*}
\end{lemma}
\begin{proof}
Since there is no other ways to describe the word $E_1E_2E_0g_{n,1}$ in terms of a product 
of words $H_1$ and $H_2$ such that $H_1$ is in $Q_{n-1}$ and $H_2$ is either $g_{n,i}$ or $h_{n,j}$,
we have 
\begin{align*}
\mathrm{coef}^{n}(E_1E_2E_0 g_{n,1})=\mathrm{coef}^{n-1}(E_1E_2E_0)\mathrm{coef}(g_{n,1})
=\genfrac{}{}{}{}{[n][n-1][n-2]}{[2n][2n-2][2]},
\end{align*}
where we have used Lemma \ref{lemma:E120} and Eq. (\ref{eq:coefgh}).
\end{proof}

We are ready to prove Proposition \ref{prop:E013}.

\begin{proof}[Proof of Proposition \ref{prop:E013}]
We rewrite $E_{0}E_{1}E_{3}$ as a product $H_1H_2$ of two words $H_1$ and $H_2$ such that 
$H_1$ is in $Q_{n-1}$ and $H_2$ is either $g_{n,i}$ or $h_{n,j}$.
Below, we list up all the possibilities of a pair of two elements $H_1$ and $H_{2}$.

Suppose that $H_{1}=E_0E_1\overline{g_{n-1,3}}$. Then, $H_2$ is either $g_{n,1}$, $h_{n,0}$ or $g_{n,3}$.

Similarly, if $H_{2}$ is either $h_{n,1}$ or $h_{n,3}$, then 
$H_1$ is either $E_{0}\overline{g_{n-1,3}}$, $E_{1}\overline{g_{n-1,3}}$, $E_{0}E_1\overline{g_{n-1,3}}$, $\overline{g_{n-1,1}}E_0$, 
$\overline{h_{n-1,0}}E_{1}$, $\overline{g_{n-1,3}}E_1E_2E_0$, $\overline{g_{n-1,3}}E_{0}E_2E_{1}$, 
$\overline{g_{n-1,1}}E_0E_2E_1$, or $\overline{h_{n-1,0}}E_1E_2E_0$.

Finally, if $H_2$ is the identity, then $H_{1}=E_{0}E_{1}E_{3}$.

Note that if $H_1=E_{0}E_{1}\overline{g_{n-1,3}}$ and $H_{2}=h_{n,1}$ or $h_{n,3}$, then we have 
$H_1H_2=-[2]E_{0}E_1E_{3}$.

By Eq. (\ref{eq:Q2}), we have 
\begin{align*}
\mathrm{coef}^{n}(E_0E_1E_3)
&=\left(\genfrac{}{}{}{}{[n-2]^2[2]}{[2n-2][n]}+\genfrac{}{}{}{}{[n-2]^3[2]^2}{[2n-2][2n-4]}\right) \\
&\qquad\times\left(2\genfrac{}{}{}{}{[n]}{[2n]}+\genfrac{}{}{}{}{[n][4]}{[2n][2]}-[2]\genfrac{}{}{}{}{[n]^2}{[2n][n+1]}
-[2]\genfrac{}{}{}{}{[n][n-2]}{[2n][n+1]}\right) \\
&\quad+2\left(\genfrac{}{}{}{}{[n]^2}{[2n][n+1]}+\genfrac{}{}{}{}{[n][n-2]}{[2n][n+1]}\right) \\
&\qquad\times
\left(\genfrac{}{}{}{}{[n-1]^2[n-2][3]}{[2n-2][2n-4][2]}+\genfrac{}{}{}{}{[n-1]^2[n-2]}{[2n-2][2n-4][2]}\right.\\
&\qquad\qquad+\left.\genfrac{}{}{}{}{[n-1][n-2][n-3][3]}{[2n-2][2n-4][2]}+\genfrac{}{}{}{}{[n-1][n-2][n-3]}{[2n-2][2n-4][2]}
\right) \\
&\quad+\mathrm{coef}^{n-1}(E_0E_1E_3), \\
&=\sum_{k=2}^{n}\left(
\genfrac{}{}{}{}{[k-2]^2[2]^3}{[2k-2][k][k+1]}+\genfrac{}{}{}{}{[k-2]^3[2]^4}{[2k-2][2k-4][k+1]}
+2\genfrac{}{}{}{}{[k][k-1]^2[k-2]^2[2]^3}{[2k][2k-2][2k-4][k+1]}
\right), \\
&=\genfrac{}{}{}{}{[n-1]^3[n-2][2]^2}{[2n-2][n+1][n]}+2\genfrac{}{}{}{}{[n][n-1]^3[n-2][2]}{[2n][2n-2][n+1]},
\end{align*}
where we have used Lemmas \ref{lemma:E120}, \ref{lemma:En310}, \ref{lemma:gn31}, \ref{lemma:0gn1},
\ref{lemma:021gn3}, \ref{lemma:120gn1} and Proposition \ref{prop:barE}. 
\end{proof}

\appendix
\section{The Jones--Wenzl projection \texorpdfstring{$Q_{3}$}{Q3}}
\label{app:A}
The Jones--Wenzl projection $Q_{3}$ for $\mathrm{TL}^{D}_{4}$ is 
given as follows:
\begin{align*}
Q_{3}
&=1+\genfrac{}{}{}{}{[3][4]}{[6][2]}(E_0+E_1+E_3)+\genfrac{}{}{}{}{[3][2]^2}{[6]}E_{2}
+\genfrac{}{}{}{}{[3]^3}{[6][4]}(E_{1}E_0+E_3E_0+E_3E_1) \\
&\quad+\genfrac{}{}{}{}{[3][2]}{[6]}(E_2E_0+E_2E_1+E_0E_2+E_1E_2+E_3E_2+E_2E_3) \\
&\quad+\genfrac{}{}{}{}{[3]^2[2]}{[6][4]}(E_2E_1E_0+E_2E_3E_0+E_2E_3E_1+E_1E_0E_2+E_3E_0E_2+E_3E_1E_2) \\
&\quad+\genfrac{}{}{}{}{[3]}{[6]}(E_1E_2E_0+E_3E_2E_0+E_0E_2E_1+E_3E_2E_1+E_0E_2E_3+E_1E_2E_3)\\
&\quad+\genfrac{}{}{}{}{[2]^3([2]^2+1)}{[6][4]}E_3E_1E_0\\
&\quad+\genfrac{}{}{}{}{[3]^2}{[6][4]}(E_3E_2E_1E_0+E_3E_1E_2E_0+E_1E_2E_3E_0+E_0E_3E_2E_1+E_0E_2E_3E_1+E_1E_0E_2E_3)\\
&\quad+\genfrac{}{}{}{}{[3][2]^2}{[6][4]}(E_2E_1E_0E_2+E_2E_3E_0E_2+E_2E_3E_1E_2) \\
&\quad+\genfrac{}{}{}{}{[2]^2([2]^2+1)}{[6][4]}(E_2E_3E_1E_0+E_3E_1E_0E_2)\\
&\quad+\genfrac{}{}{}{}{[3][2]}{[6][4]}(E_2E_3E_1E_2E_0+E_2E_0E_3E_2E_1+E_3E_2E_1E_0E_2+E_1E_2E_3E_0E_2\\
&\qquad+E_0E_2E_3E_1E_2+E_2E_1E_0E_2E_3) \\
&\quad+\genfrac{}{}{}{}{[2]([2]^2+1)}{[6][4]}E_2E_3E_1E_0E_2 \\
&\quad+\genfrac{}{}{}{}{[3]}{[6][4]}(E_0E_2E_3E_1E_2E_0+E_1E_2E_3E_0E_2E_1+E_3E_2E_1E_0E_2E_3).	
\end{align*}

\bibliographystyle{amsplainhyper} 
\bibliography{biblio}

\end{document}